\documentclass[reqno]{amsart}
\usepackage{bbm}
\usepackage{}
\usepackage{stmaryrd}
\usepackage{mathrsfs}
\usepackage{cases}
\usepackage{amsfonts}
\usepackage{amssymb}
\usepackage{amsmath}
\usepackage{tikz}
\usepackage{extarrows}
\allowdisplaybreaks[4]
\newtheorem{theorem}{Theorem}[section]
\newtheorem{lemma}[theorem]{Lemma}
\newtheorem{proposition}[theorem]{Proposition}

\theoremstyle{definition}
\newtheorem{definition}[theorem]{Definition}

\newtheorem{remark}[theorem]{Remark}
\numberwithin{equation}{section}

\begin{document}
\title[Embedding between Local Hardy and $\alpha$-Modulation Spaces]{Sharp embedding relations between Local Hardy and $\alpha$-Modulation Spaces}
\author{GUOPING ZHAO}
\address{School of Applied Mathematics, Xiamen University of Technology, Xiamen, 361024, P.R.China}
\email{guopingzhaomath@gmail.com}

\author{GUILIAN GAO}
\address{School of Science, Hangzhou Dianzi University, Hangzhou, 310018, P.R. China}
\email{gaoguilian@hdu.edu.cn}

\author{WEICHAO GUO}
\address{School of Mathematics and Information Sciences, Guangzhou University, Guangzhou, 510006, P.R.China}
\email{weichaoguomath@gmail.com}


\begin{abstract}
  We give the optimal embedding relations between local Hardy space and $\alpha$-modulation spaces,
  which extend the results for the embedding relations between local Hardy and modulation spaces obtained by Kobayashi, Miyachi and Tomita in [Studia Math. 192 (2009), 79-96].
\end{abstract}

\maketitle

\section{Introduction}
The modulation space $M_{p,q}^{s}$ was first introduced by Feichtinger \cite{Feichtinger} in 1983.
Some motivations and historical remarks can be founded in \cite{Feichtinger_Survey}.
As function spaces associated with the uniform decomposition (see \cite{Tribel_modulation-1983}),
modulation space has a close relationship with the topics of time-frequency analysis (see \cite{Grochenig}),
and it has been regarded as a appropriate space for the study of partial differential equations (see \cite{RSW_2012}).
On the other hand, it is well known that the Besov space $B_{p,q}^{s}$, constructed by the dyadic
decomposition on frequency plane, is also a popular function space in the studies of harmonic analysis and partial differential equations.

A more general function space, named $\alpha$-modulation space, denoted by $M_{p,q}^{s,\alpha}$, was  introduced by Gr\"{o}bner \cite{Grobner-doctor+thesis-1992} in $1992$,
in order to link modulation and Besov spaces by the parameter $\alpha\in [0,1]$.
The modulation spaces $M_{p,q}^s$ are the special $\alpha$-modulation spaces when $\alpha=0$,
and the (inhomogeneous) Besov space $B_{p,q}^s$ can be regarded as the limit case of $M_{p,q}^{s,\alpha}$ as $\alpha\rightarrow 1$ (see \cite{Grobner-doctor+thesis-1992}).
Sometimes, we use $M_{p,q}^{s,1}$ to denote the inhomogeneous Besov space $B_{p,q}^s$ for convenience.
The interested reader can find some basic properties of $\alpha$-modulation spaces in \cite{Han-Wang_embedding-JMSJ-2014}.
We note that the $\alpha$-modulation space is not the  intermediate space between modulation and Besov spaces in the sense of complex interpolation (see \cite{Guo-Zhao_interpolation}).

Embedding relations among frequency decomposition spaces, including Lebesgue spaces, Sobolev spaces, Besov spaces, modulation spaces and $\alpha$-modulation spaces, have been concerned
by many authors, for example, one can see
Gr\"{o}bner \cite{Grobner-doctor+thesis-1992},
Okoudjou \cite{Okoudjou-Proceeding-2004},
Toft \cite{Toft_Continunity-JFA-2004},
Sugimoto-Tomita \cite{Sugimoto-Tomita-JFA-2007},
Wang-Han \cite{Han-Wang_embedding-JMSJ-2014},
Guo-Fan-Wu-Zhao \cite{Guo-Fan-Wu-Zhao_embedding}
for embedding relation among modulation spaces, Besov spaces and $\alpha$-modulation spaces.
In particular, Kobayashi-Miyachi-Tomita \cite{Kobayashi-Miyachi-Tomita-Studia-2009} studied the embedding relations between local Hardy spaces $h_p$ and modulation spaces $M_{p,q}^s$,
Kobayashi-Sugimoto \cite{Kobayashi-Sugimoto-JFA-2011} consider the embedding relations between Sobolev and modulation spaces.
We state their results as follows.

\bigskip
For $(p,q)\in (0,\infty]^2$, we define $A(p,q)=0\wedge n(1-1/p-1/q)\wedge n(1/p-1/q)$, $B(p,q)=0\vee  n(1-1/p-1/q)\vee n(1/p-1/q)$, that is,
\begin{equation*}
  A(p,q)=
  \begin{cases}
  0,  &\text{if } (1/p,1/q)\in A_1:\, 1/q\leqslant (1-1/p)\wedge 1/p,\\
  n(1-1/p-1/q), &\text{if } (1/p,1/q)\in A_2:\, 1/p\geqslant (1-1/q)\vee 1/2,\\
  n(1/p-1/q), &\text{if } (1/p,1/q)\in A_3:\, 1/p \leqslant 1/q \wedge 1/2;
  \end{cases}
\end{equation*}
\begin{equation*}
  B(p,q)=
  \begin{cases}
  0,  &\text{if } (1/p,1/q)\in B_1:\, 1/q\geqslant (1-1/p)\vee 1/p,\\
  n(1-1/p-1/q), &\text{if } (1/p,1/q)\in B_2:\, 1/p\leqslant (1-1/q)\wedge 1/2, \\
  n(1/2-1/q), &\text{if } (1/p,1/q)\in B_3:\, 1/p\geqslant 1/q\vee 1/2.
  \end{cases}
\end{equation*}

\begin{tikzpicture}[scale=3]
\coordinate (Origin) at (0,0);
\coordinate(horizontal_axis) at (1.2,0);
\coordinate(vertical_axis) at (0,1.2);
\coordinate (m10) at (1,0);
\coordinate (m0505) at (0.5,0.5);
\coordinate (m11) at (1.2,1.2);
\coordinate (m051) at (0.5,1.2);
\draw[thick,->] (Origin) -- (horizontal_axis) node[below]{$\frac{1}{p_1}$};
\draw[thick,->] (Origin) -- (vertical_axis) node[left]{$\frac{1}{q_2}$};
\draw [very thick] (m0505)-- (m10) node [below] {$1$};
\draw [very thick] (m0505)--(Origin) node [left] {$(0,0)$};
\draw [very thick] (m0505)-- (m051);
\draw(0.8,0.7)node{$A_2$};
\draw(0.5,0.2)node{$A_1$};
\draw(0.25,0.5)node{$A_3$};
\node at (0.6,-0.25){\small Figure 1};
\coordinate (2Origin) at (2,0);
\coordinate(2horizontal_axis) at (3.2,0);
\coordinate(2vertical_axis) at (2,1.2);
\coordinate (2m050) at (2.5,0);
\coordinate (2m0505) at (2.5,0.5);
\coordinate (2m11) at (3.2,1.2);
\coordinate (2m01) at (2,1);
\draw[thick,->] (2Origin) -- (2horizontal_axis) node[below]{$\frac{1}{p_2}$};
\draw[thick,->] (2Origin) -- (2vertical_axis) node[left]{$\frac{1}{q_1}$};
\draw [very thick] (2m0505)-- (2m01) node [left] {$1$};
\draw [very thick] (2m11)--(2m0505);
\draw (2Origin) node [left] {$(0,0)$};
\draw [very thick] (2m0505)-- (2m050) node [below] {$1/2$};
\draw(2.7,0.4)node{$B_3$};
\draw(2.25,0.3)node{$B_2$};
\draw(2.5,0.8)node{$B_1$};
\node at (2.6,-0.25){\small Figure 2};
\end{tikzpicture}

\hspace{-12pt}\textbf{Theorem A} (cf. \cite{Kobayashi-Miyachi-Tomita-Studia-2009})\quad
Let $0<p\leqslant 1$, $0<q\leqslant \infty$ and $s\in \mathbb{R}$.
Then, $h_{p} \subset M_{p,q}^s$
  if and only if $s\leqslant A(p,q)$ with strict inequality when $1/q> 1/p$.
\bigskip
~\\
\hspace{-12pt}\textbf{Theorem B} (cf. \cite{Kobayashi-Miyachi-Tomita-Studia-2009})\quad
Let $0<p\leqslant 1$, $0<q\leqslant \infty$ and $s\in \mathbb{R}$.
Then, $M_{p,q}^{s} \subset h_{p}$
  if and only if $s\geqslant B(p,q)$ with strict inequality when $1/p> 1/q$.
~\\
\hspace{-12pt}\textbf{Theorem C} (cf. \cite{Kobayashi-Sugimoto-JFA-2011})\quad
Let $1\leqslant p,q \leqslant \infty$ and $s\in \mathbb{R}$.
Then, $L^p \subset M_{p,q}^s$
  if and only if one of the following conditions holds:
    \begin{enumerate}
     \item
     $p>1$, $s\leqslant A(p,q)$ with strict inequality when $1/q> 1/p$,
     \item
     $p=1$, $s\leqslant A(1,q)$ with strict inequality when $q\neq \infty$.
   \end{enumerate}
~\\
\hspace{-12pt}\textbf{Theorem D} (cf. \cite{Kobayashi-Sugimoto-JFA-2011})\quad
Let $1\leqslant p,q \leqslant \infty$ and $s\in \mathbb{R}$.
Then, $M_{p,q}^{s} \subset L^p$
  if and only if one of the following conditions holds:
    \begin{enumerate}
     \item
     $p<\infty$, $s\geqslant B(p,q)$ with strict inequality when $1/p> 1/q$,
     \item
     $p=\infty$, $s\geqslant B(p,q)$ with strict inequality when $q>1$.
   \end{enumerate}

As mentioned before, $\alpha$-modulation space is a generalization of modulation space, and that can not be obtained by interpolation between
modulation and Besov spaces,
so it is interesting to ask what is the sharp conditions for the embedding between $\alpha$-modulation and local Hardy spaces (Sobolev spaces).
The main purpose of this paper is to give the answer.
Now, we state our main theorems as follows.

\begin{theorem} \label{theorem, sharpness of embedding from local Hardy spaces to alpha-modulation spaces}
Let $\alpha\in[0,1)$, $0< p_1<\infty$, $0<p_2, q_2\leqslant\infty$, $s_2\in \mathbb{R}$. Then
\begin{equation}\label{conclusion, local Hardy spaces embedded in alpha-modulation spaces}
h^{p_1}\subset M^{s_2,\alpha}_{p_2,q_2}
\end{equation}
if and only if $1/{p_2}\leqslant1/{p_1}$ and
\begin{equation}\label{condition, from local Hardy to mod}
  s_2\leqslant n\alpha(1/p_2-1/p_1)+(1-\alpha)A(p_1,q_2),
\end{equation}
with strict inequality in (\ref{condition, from local Hardy to mod}) when $1/q_2> 1/p_1$.
\end{theorem}

\begin{theorem}\label{theorem, sharpness of embedding from alpha-modulation spaces to local Hardy spaces}
Let $\alpha\in[0,1)$, $0< p_2<\infty$, $0<p_1, q_1\leqslant\infty$, $s_1\in \mathbb{R}$. Then
\begin{equation}\label{conclusion, alpha-modulation spaces embedded in local Hardy spaces}
M^{s_1,\alpha}_{p_1,q_1}\subset h^{p_2}
\end{equation}
if and only if $1/{p_2}\leqslant1/{p_1}$ and
\begin{equation}\label{condition, from mod to local Hardy}
  s_1\geqslant n\alpha(1/p_1-1/p_2)+(1-\alpha)B(p_2,q_1),
\end{equation}
with strict inequality in (\ref{condition, from mod to local Hardy}) when $1/p_2> 1/q_1$.
\end{theorem}

\begin{theorem}\label{theorem, sharpness of embedding from L1 to alpha-modulation spaces}
  Let $\alpha\in[0,1)$, $0< p, q\leqslant\infty$, $s\in \mathbb{R}$. Then
  \begin{equation}\label{conclusion, L1 embedded in alpha-mod.}
  L^1\subset M^{s,\alpha}_{p,q}
  \end{equation}
  if and only if $1/{p} \leqslant1$ and
  \begin{equation}\label{condtion, L1 embedded in alpha mod.}
    s\leqslant n\alpha(1/p-1)+n(1-\alpha)A(1,q)
  \end{equation}
  with the strict inequality in (\ref{condtion, L1 embedded in alpha mod.}) when $q<\infty$.
\end{theorem}
\begin{theorem}\label{theorem, sharpness of embedding from alpha-modulation spaces to L1}
  Let $\alpha\in[0,1)$, $0< p, q\leqslant\infty$, $s\in \mathbb{R}$. Then
  \begin{equation}\label{conclusion, alpha-mod. embedded in L1}
  M^{s,\alpha}_{p,q}\subset L^1
  \end{equation}
  if and only if $1/{p} \geqslant1$, and
  \begin{equation}\label{condition, from mod to L1}
  s\geqslant n\alpha(1/p-1)+(1-\alpha)B(1,q)
  \end{equation}
  with the strict inequality in (\ref{condition, from mod to L1}) when $q>1$.
\end{theorem}
\begin{theorem}\label{theorem, sharpness of embedding from L-infinity to alpha-modulation spaces}
  Let $\alpha\in[0,1)$, $0< p, q\leqslant\infty$, $s\in \mathbb{R}$. Then
  \begin{equation}\label{conclusion, L-infty embedded in alpha-mod.}
  L^{\infty}\subset M^{s,\alpha}_{p,q}
  \end{equation}
  if and only if $p=\infty$ and
  \begin{equation}\label{condtion, L-infty embedded in alpha-mod.}
    s\leqslant n(1-\alpha)A(\infty,q)
  \end{equation}
  with the strict inequality in (\ref{condtion, L-infty embedded in alpha-mod.}) when $q<\infty$.
\end{theorem}
\begin{theorem}\label{theorem, sharpness of embedding from alpha-modulation spaces to L-infinity}
  Let $\alpha\in[0,1)$, $0< p, q\leqslant\infty$, $s\in \mathbb{R}$. Then
  \begin{equation}\label{conclusion, alpha-mod. embedded in L-infty}
  M^{s,\alpha}_{p,q}\subset L^{\infty}
  \end{equation}
  if and only if
  \begin{equation*}
  s\geqslant n\alpha/p+n(1-\alpha)B(\infty,q),
  \end{equation*}
  with strict inequality when $q>1$.
\end{theorem}

\begin{remark}
  Obviously, our theorems generalize the main results in \cite{Kobayashi-Miyachi-Tomita-Studia-2009} and \cite{Kobayashi-Sugimoto-JFA-2011}.
  If we choose $p_1=p_2$ and $\alpha=0$ in our theorems, we regain the results obtained in \cite{Kobayashi-Miyachi-Tomita-Studia-2009, Kobayashi-Sugimoto-JFA-2011}.
  Furthermore, we use a quite different method to achieve our goals.
  For the necessity part, we reduce the embedding between local Hardy and $\alpha$-modulation spaces
  to the corresponding embedding associated with discrete sequences.
  For the sufficiency, we estimate the $h^p$ atoms under the "standard" norm of $\alpha$-modulation spaces.
  Our method also works well in the special case $\alpha=0$.
  Without using an equivalent norm of modulation space (see Lemma 2.2 in \cite{Kobayashi-Miyachi-Tomita-Studia-2009}),
  our method seems more readable and efficient.
  We also remark that a similar embedding problem (the case $p_1=p_2$) is studied by T. Kato in his doctoral thesis \cite{Kato_thesis}, in which he use the same method as in \cite{Kobayashi-Miyachi-Tomita-Studia-2009}.
  Due to the generality of $p_1, p_2$ and the quite different method, our work has its own interesting.
\end{remark}

\begin{remark}
  Recalling $h^p\sim L^p$ for $p\in (1,\infty)$. Except for some endpoints, it is convenient to deal with the case $p\leqslant 1$ and $p\geqslant 1$ together.
  Thus, we establish Theorem \ref{theorem, sharpness of embedding from local Hardy spaces to alpha-modulation spaces} and Theorem \ref{theorem, sharpness of embedding from alpha-modulation spaces to local Hardy spaces}
  for all the embedding between $\alpha$-modulation and local Hardy spaces in the full range. Then, we use Theorem \ref{theorem, sharpness of embedding from L1 to alpha-modulation spaces}
  to \ref{theorem, sharpness of embedding from alpha-modulation spaces to L-infinity} to deal with the endpoint cases.
\end{remark}

\begin{remark}
  As mentioned before, Besov space can be regarded as the special $\alpha$-modulation space with $\alpha=1$.
  However, the embedding relations with $\alpha=1$ are the special cases of embedding between Besov and Triebel-Lizorkin spaces, which have been fully studied.
  So we focus on $\alpha\in [0,1)$ in this paper.
\end{remark}

Our paper is organized as follows. In Section 2, we give some
definitions of function spaces treated in this paper. We also collect some
basic properties which is useful in our proof.
Section 3 is prepared for some estimates of $h^p$-atoms, which lead to the key embedding relations of our theorems.
The proofs of our main results will be given in Section 4 to 7.

\section{Preliminaries and Definitions}
Let $\mathscr{S}:= \mathscr{S}(\mathbb{R}^{n})$ be the Schwartz space
and $\mathscr{S}':=\mathscr{S}'(\mathbb{R}^{n})$ be the space of tempered distributions.
Let $\mathbb{N}_0:=\mathbb{Z}^{+}\bigcup\{0\}$ be the collection of all nonnegative integers.
We define the Fourier transform $\mathscr {F}f$ and the inverse Fourier transform $\mathscr {F}^{-1}f$ of $f\in \mathscr{S}(\mathbb{R}^{n})$ by
$$
\mathscr {F}f(\xi)=\hat{f}(\xi)=\int_{\mathbb{R}^{n}}f(x)e^{-2\pi ix\cdot \xi}dx
,
~~
\mathscr {F}^{-1}f(x)=\hat{f}(-x)=\int_{\mathbb{R}^{n}}f(\xi)e^{2\pi ix\cdot \xi}d\xi.
$$

The notation $X\lesssim Y$ denotes the statement that $X\leqslant CY$, with a positive constant $C$ that may depend on $n,\alpha, \,p_{i},\,q_{i},\,s_{i}(i=1,\,2)$,
but it might be different from line to line.
The notation $X\sim Y$ means the statement $X\lesssim Y\lesssim X$.
And the notation $X\simeq Y$ stands for $X=CY$.
For a multi-index $k=(k_{1},k_{2},...,k_{n})\in \mathbb{Z}^{n}$,
we denote $|k|_{\infty }:=\max\limits_{i=1,2,...,n}|k_{i}|$ and $\langle k\rangle:=(1+|k|^{2})^{{1}/{2}}.$

We recall some definitions of the function spaces treated in this paper.

First we give the partition of unity on frequency space associated with $\alpha\in [0,1)$.
Take two appropriate constants $c>0$ and $C>0$ and choose a Schwartz function sequence $\{\eta_k^{\alpha}\}_{k\in \mathbb{Z}^n}$
satisfying
\begin{equation}\label{decompositon function of alpha modulation spaces}
\begin{cases}
|\eta_k^{\alpha}(\xi)|\geqslant 1,   \text{~if~} |\xi-\langle k\rangle^{\frac{\alpha}{1-\alpha}}k|<c\langle k\rangle^{\frac{\alpha}{1-\alpha}};\\
\text{supp}\eta_k^{\alpha}\subset \{\xi\in\mathbb{R}^n: |\xi-\langle k\rangle^{\frac{\alpha}{1-\alpha}}k|<C\langle k\rangle^{\frac{\alpha}{1-\alpha}}\};\\
\sum_{k\in \mathbb{Z}^{n}}\eta_k^{\alpha}(\xi)\equiv 1, \forall \xi\in \mathbb{R}^{n};\\
|\partial^{\gamma}\eta_k^{\alpha}(\xi)|\leqslant C_{|\alpha|}\langle k\rangle^{-\frac{\alpha|\gamma|}{1-\alpha}} , \forall \xi\in \mathbb{R}^{n}, \gamma \in(\mathbb{Z}^{+}\cup\{0\})^{n}.
\end{cases}
\end{equation}
The sequence $\{\eta_{k}^{\alpha}\}_{k\in\mathbb{Z}^{n}}$ constitutes a smooth decomposition of $\mathbb{R}^{n}$.
The  frequency decomposition operators associated with the above function sequence are defined by
\begin{equation}
\Box_{k}^{\alpha}:= \mathscr {F}^{-1}\eta_{k}^{\alpha}\mathscr{F}
\end{equation}
for $k\in \mathbb{Z}^{n}$.
Let $0< p,q \leqslant \infty$, $s\in \mathbb{R}$, $\alpha \in [0,1)$. Then the $\alpha$-modulation space associated with above decomposition is defined by
\begin{equation}
M_{p,q}^{s,\alpha}(\mathbb{R}^n)=\{f\in \mathscr {S}'(\mathbb{R}^{n}): \|f\|_{M_{p,q}^{s,\alpha}(\mathbb{R}^n)}
=\left( \sum_{k\in \mathbb{Z}^{n}}\langle k\rangle ^{\frac{sq}{1-\alpha}}\|\Box_k^{\alpha} f\|_{L^p}^{q}\right)^{{1}/{q}}<\infty \}
\end{equation}
with the usual modifications when $q=\infty$. For simplicity, we write $M_{p,q}^s=M_{p,q}^{s,0}$, $M_{p,q}=M_{p,q}^{0,0}$ and $\eta_k(\xi)=\eta_k^0(\xi)$.

\begin{remark}
We recall that the above definition is independent of the choice of exact $\eta_k^{\alpha}$ (see \cite{Han-Wang_embedding-JMSJ-2014}).
Also, for sufficiently small $\delta>0$,
one can construct a function sequence $\{\eta_{k}^{\alpha}\}_{k\in\mathbb{Z}^{n}}$ such that
$\eta_k^{\alpha}(\xi)=1$, and $\eta_k^{\alpha}(\xi)\eta_l^{\alpha}(\xi)=0$ if $k\neq l$
and $\xi$ lies in the ball $B(\langle k\rangle^{\frac{\alpha}{1-\alpha}}k,\langle k\rangle^{\frac{\alpha}{1-\alpha}}\delta)$ (see \cite{Forasier, Guo_Chen}).
\end{remark}

\begin{lemma}[Embedding of $\alpha$-modulation spaces]\label{Lemma, embedding of alpha modulation spaces}
  Let $0<p_1,\,p_2,\,q_1,\,q_2\leqslant \infty$, $s_1,\,s_2\in \mathbb{R}$, $\alpha\in[0,1)$.
  Then $M_{p_1,q_1}^{s_1,\alpha}\subset M_{p_2,q_2}^{s_2,\alpha}$ if and only if
  \begin{equation}
    \begin{cases}
      \frac{1}{p_2}\leqslant\frac{1}{p_1}\\
      \frac{1}{q_2}\leqslant\frac{1}{q_1}\\
      \frac{s_2}{n}-\frac{\alpha}{p_2}\leqslant\frac{s_1}{n}-\frac{\alpha}{p_1}
    \end{cases}
    \text{or}\hspace{5mm}
    \begin{cases}
      \frac{1}{p_2}\leqslant\frac{1}{p_1}\\
      \frac{1}{q_2}>\frac{1}{q_1}\\
      \frac{s_2}{n}-\frac{\alpha}{p_2}+\frac{1-\alpha}{q_2}<\frac{s_1}{n}-\frac{\alpha}{p_1}+\frac{1-\alpha}{q_1}.
    \end{cases}
  \end{equation}
\end{lemma}

The above embedding lemma can be verified by the same method used in the proof of embedding between modulation spaces. We omit the details here. One can see \cite{Guo-Fan-Wu-Zhao_embedding} for a more general case..
Next, we give some definitions and properties of sequences associated with $\alpha$.

\begin{definition}
Let $0<p,q\leqslant \infty$, $s\in \mathbb{R}$, $\alpha\in[0,1)$.
Let $\vec{a}:=\{a_k\}_{k\in\mathbb{Z}^n}$ be a sequence, we denote its $l_p^{s,\alpha}$ (quasi-)norm
\begin{numcases}{\|\vec{a}\|_{l_p^{s,\alpha}}=}
\left(\sum_{k\in \mathbb{Z}^n}|a_k|^p \langle k\rangle^{\frac{sp}{1-\alpha}}\right)^{1/p}, &$p<\infty$
\\
\sup_{k\in \mathbb{Z}^n}|a_k|\langle k\rangle^{\frac{s}{1-\alpha}}, \hspace{15mm}&$p=\infty$
\end{numcases}
and let $l_p^{s,\alpha}$ be the (quasi-)Banach space of sequences whose $l_p^{s,\alpha}$ (quasi-)norm is finite.
Let $\vec{b}:=\{b_j\}_{j\in \mathbb{N}}$ be a sequence, we denote its $l_p^{s,1}$ (quasi-)norm
\begin{numcases}{\|\vec{b}\|_{l_p^{s,1}}=}
\left(\sum_{j\in \mathbb{N}_0}|b_j|^p 2^{jsp}\right)^{1/p}, &$p<\infty$
\\
\sup_{j\in \mathbb{N}_0}|b_j|2^{js}, \hspace{15mm}&$p=\infty$
\end{numcases}
and let $l_p^{s,1}$ be the (quasi-)Banach space of sequences whose $l_p^{s,1}$ (quasi-)norm is finite.
\end{definition}

We also give the following lemma for the embedding about sequences defined above. The proof is quite standard, we omit the details here.

\begin{lemma}[Sharpness of embedding, for $\alpha$-decomposition] \label{lemma, Sharpness of embedding, for alpha decomposition}
Suppose $0<q_1, q_2\leqslant \infty$, $s_1, s_2\in \mathbb{R}$, $\alpha\in [0,1]$. Then
\begin{equation}
l_{q_1}^{s_1,\alpha}\subset l_{q_2}^{s_2,\alpha}
\end{equation}
holds if and only if
\begin{equation}
(1-\alpha)/q_2+s_2/n<(1-\alpha)/q_1+s_1/n
\hspace{10mm}
\text{or} \hspace{10mm}
\begin{cases}
s_2= s_1,\\
1/q_2\leqslant 1/q_1.
\end{cases}
\end{equation}
\end{lemma}

To define the Besov spaces and Triebel-Lizorkin spaces, we introduce  the dyadic decomposition of $\mathbb{R}^n$.
Let $\varphi$ be a smooth bump function supported in the ball $\{\xi: |\xi|<\frac{3}{2}\}$ and be equal to 1 on the ball $\{\xi: |\xi|\leqslant \frac{4}{3}\}$.
Denote
\begin{equation}\label{decompositon function of Besov space}
\phi(\xi)=\varphi(\xi)-\varphi(2\xi),
\end{equation}
and a function sequence
\begin{equation}
\begin{cases}
\phi_j(\xi)=\phi(2^{-j}\xi),~j\in \mathbb{Z}^{+},
\\
\phi_0(\xi)=1-\sum\limits_{j\in \mathbb{Z}^+}\phi_j(\xi)=\varphi(\xi).
\end{cases}
\end{equation}
For integers $j\in \mathbb{N}_0$, we define the Littlewood-Paley operators
\begin{equation}
\Delta_j=\mathscr{F}^{-1}\phi_j\mathscr{F}.
\end{equation}
Let $0<p,q\leqslant \infty$, $s\in \mathbb{R}$. For a tempered distribution  $\ f$,  we set the norm
\begin{equation}
\|f\|_{B_{p,q}^s}=\left(\sum_{j\in\mathbb{N}_0}2^{jsq}\|\Delta_jf\|_{L^p}^q \right)^{1/q},~
\end{equation}
with the usual modifications when $q=\infty$.
The (inhomogeneous) Besov  space ${B_{p,q}^s}$ is the space of all tempered distributions $f$ for which the quantity $\|f\|_{B_{p,q}^s}$ is finite.
Let $0<p<\infty$, $0<q\leqslant \infty$, $s\in \mathbb{R}$. For a tempered distribution  $\ f$,  we set the norm
\begin{equation}
\|f\|_{F_{p,q}^s}=\left\|\left(\sum_{j\in\mathbb{N}_0}2^{jsq}|\Delta_jf|^q \right)^{1/q}\right\|_{L^p}
\end{equation}
with the usual modifications when $q=\infty$.
The Triebel-Lizorkin space ${F_{p,q}^s}$ is the space of all tempered distributions $f$ for which the quantity $\|f\|_{F_{p,q}^s}$ is finite.

%
%
Now, we turn to introduce the local Hardy space of Goldberg \cite{Goldberg_hp-Duke-1979}.
Let $\psi\in\mathscr{S}$ satisfy $\int_{\mathbb{R}^n}\psi(x)dx\neq0$.
Define $\psi_t=t^{-n}\psi(x/t)$.
The \emph{local Hardy spaces} is defined by
\begin{equation*}
  h^p:=\{f\in\mathscr{S}': \|f\|_{h^p}=\|\sup_{0<t<1}|\psi_t\ast f|\|_{L^p}<\infty\}.
\end{equation*}
We note that the definition of the local Hardy spaces is independent of the choice of $\psi\in \mathscr{S}$.
A function $a$ is said to be an \textit{$h^p$-atom of type $I$} (small $h^p$-atom)
if it satisfies the following \textit{support condition}, \textit{size condition} and \textit{vanishing moment condition}:
  \begin{equation}
    \begin{split}
    \text{supp~}a\subset Q \text{ with } |Q|<1, \|a\|_{L^{\infty}}\leqslant |Q|^{-1/p},\\
    \int{x^{\beta}a(x)dx}=0 \text{ for all } |\beta|\leqslant [n(1/p-1)],
    \end{split}
  \end{equation}
  where $Q$ is a cube and $|Q|$ is the Lebesgue measure of $Q$, and $[n(1/p-1)]$ is the integer part of $n(1/p-1)$.
A function $a$ is said to be an \textit{$h^p$-atom of type $I\!I$} (big $h^p$-atom) if it satisfies
  \begin{equation}
    \text{supp~}a\subset Q \text{ with } |Q|\geqslant1, \|a\|_{L^{\infty}}\leqslant |Q|^{-1/p}.
  \end{equation}
All the big and small $h_p$-atoms are collectively called $h_p$-atom.
We recall $\|a\|_{h_p}\lesssim 1$ for all $h_p$-atoms. On the other hand, any $f\in h_p (p\leqslant 1)$ can be represented by
\begin{equation}
  f=\sum_{j=1}^{\infty}\lambda_ja_j,
\end{equation}
where the series converges in the sense of distribution, $\{a_j\}$ is a collection of $h_p$-atoms and $\{\lambda_j\}$ is a sequence of complex numbers such that $\sum\limits_{j=1}^{\infty}|\lambda_j|^p<\infty$.
Moreover, we have
\begin{equation}
  \|f\|_{h^p}\sim\inf\left( \sum_{j=1}^{\infty}|\lambda_j|^p \right)^{1/p},
\end{equation}
where the infimum is taken over all representations $f=\sum\limits_{j=1}^{\infty}\lambda_ja_j$(see Lemma 5 in \cite{Goldberg_hp-Duke-1979}).

We recall that the local Hardy space $h^p$ is equivalent with the inhomogeneous Triebel-Lizorkin space $F_{p,2}^0$ for $p\in(0,\infty)$ (see 1.4 in \cite{Tribel_92}).
For the sake of convenience, we usually use $F_{p,2}^0$ as the norm of local Hardy space throughout the remainder of this article.

\begin{lemma}[Young's inequality]\label{Young inequality}~
 \begin{enumerate}
   \item Let $0<p\leqslant 1$, $R>0$, $\mathrm{supp}{\hat{f}},~\mathrm{supp}{\hat{g}}\subseteq B(x,R)\subseteq \mathbb{R}^n$.
   We have
   \begin{equation}
   \|f*g\|_{L^p} \leqslant C R^{n(\frac{1}{p}-1)} \|f\|_{L^p} \|g\|_{L^p},
   \end{equation}
   for all $f,g\in \mathscr{S}(\mathbb{R}^n)$ and $R>0$, where $C$ is independent of $x$, $x\in \mathbb{R}^n$.
   \item Let $1\leqslant p, q, r\leqslant \infty$ satisfy $1+{1}/{q}={1}/{p}+{1}/{r}$. Then we have
   \begin{equation}
   \|f*g\|_{L^q} \lesssim \|f\|_{L^p} \|g\|_{L^r}.
   \end{equation}
 \end{enumerate}

\end{lemma}



The following proposition shows that the local Hardy space $h^p$ is equivalent with $L^p$ in the local meaning.
\begin{proposition}\label{equivlent of Lp and hp with Fourier support}
  Let $0<p<\infty$, $\alpha \in [0,1)$. Suppose $\{\varphi_k\}_{k\in \mathbb{Z}^n}$ is a sequence of smooth functions such that
  \begin{equation*}
    \text{supp}\varphi_k\subset B(\langle k\rangle^{\frac{\alpha}{1-\alpha}}k, C\langle k\rangle^{\frac{\alpha}{1-\alpha}})
  \end{equation*}
   where $C$ is a fixed positive constant. Then we have
  $$\|\varphi_k\|_{h^p}\sim \|\varphi_k\|_{L^p}$$
for all $k\in \mathbb{Z}^n$.
\end{proposition}
\begin{proof}
  By the definition of local Hardy space, we have $\|\varphi_k\|_{h^p}\gtrsim \|\varphi_k\|_{L^p}$.
  For the opposite direction, we use the Littlewood-Paley characterization of the local Hardy space.
  By Lemma \ref{Young inequality}, we have
  \begin{equation}
    \begin{split}
      \|\varphi_k\|_{h^p}
      =\|\varphi_k\|_{F^{0}_{p,2}}
      =\left\|\left(\sum_{j\in \mathbb{Z}^n}|\Delta_j\varphi_k|^2\right)^{\frac{1}{2}}\right\|_{L^{p}}
      \lesssim \sum_{j\in \mathbb{Z}^n}\left\|\Delta_j\varphi_k\right\|_{L^{p}}
      \lesssim \|\varphi_k\|_{L^p},
    \end{split}
  \end{equation}
  where we use the fact that there are only finite $j$ such that $\Delta_j\varphi_k$ is not zero.
\end{proof}


\section{Estimates of $h^p$ atoms on $\alpha$-modulation spaces}
In order to make the proof of main theorems more clear, we give some important estimates associated with $h^p$-atoms.
\begin{proposition}[Pointwise estimate for $h^p$ atoms]\label{Proposition, pointwise estiamte for atoms}
  Suppose $a$ is a function supported in a cube $Q$ centered at the origin, satisfying $\|a\|_{L^{\infty}}\leqslant |Q|^{-1/p}$. We let $Q^{*}$ be the
cube with side length $2\sqrt{n}l(Q)$ having the same center as $Q$, where the $l(Q)$ is the
side length of $Q$. Take $\rho$ to be a smooth function with compact support near the origin. Denote
$\rho_k^{\alpha}(\xi)=\rho\left(\frac{\xi-\langle k\rangle^{\frac{\alpha}{1-\alpha}}k}{\langle k\rangle^{\frac{\alpha}{1-\alpha}}}\right)$ and $\widetilde{\Box}_k^{\alpha}=\rho_k^{\alpha}(D)$.
Then we have
  \begin{equation*}
      |\widetilde{\Box}_k^{\alpha}a(x)|\lesssim \langle k \rangle^{\frac{n\alpha}{1-\alpha}}\langle\langle k \rangle^{\frac{n\alpha}{1-\alpha}}x\rangle^{-\mathscr{L}}|Q|^{1-1/p}
                                                                                                     \text{ for } x\in (Q^*)^c,
  \end{equation*}
  where $\mathscr{L}$ is a fixed number which could be arbitrary large.
  If the function $a$ also satisfies the vanishing moment condition as follow
  \begin{equation*}
    \int_{\mathbb{R}^n}{x^{\beta}a(x)dx}=0 \text{ for all } |\beta|\leqslant N:= [n(1/p-1)],
  \end{equation*}
  we also have the following estimate
  \begin{equation*}
      |\widetilde{\Box}_k^{\alpha}a(x)|\lesssim \langle k \rangle^{\frac{n\alpha}{1-\alpha}} (\sqrt[n]{|Q|}\langle k \rangle^{\frac{1}{1-\alpha}})^{N+1}
                                    \langle \langle k \rangle^{\frac{\alpha}{1-\alpha}} x\rangle^{-\mathscr{L}}|Q|^{1-1/p},
                                                                                                     \text{ for } x\in (Q^*)^c,
  \end{equation*}
  where $\mathscr{L}$ is a fixed number which could be arbitrary large.
\end{proposition}
\begin{proof}
 Using the assumption, we obtain $\|a\|_{L^1}\leqslant \|a\|_{L^{\infty}}\cdot |Q| \leqslant|Q|^{1-1/p}$.
 By the rapid decay of $\mathscr{F}^{-1}\rho$, we deduce that
  \begin{equation}
    \begin{split}
     |\widetilde{\Box}_k^{\alpha}a(x)|&=\left| \int_{\mathbb{R}^n}(\rho_k^{\alpha})^{\vee}(x-y)a(y)dy \right|
                          =\left| \int_{Q}(\rho_k^{\alpha})^{\vee}(x-y)a(y)dy \right|
                          \\
                          &\leqslant
                          \int_{Q}|\langle k \rangle^{\frac{n\alpha}{1-\alpha}}\rho^{\vee}(\langle k \rangle^{\frac{\alpha}{1-\alpha}}(x-y))a(y)|dy
                          \\
                          &\lesssim
                          \langle k \rangle^{\frac{n\alpha}{1-\alpha}}\int_{Q}|\langle\langle k \rangle^{\frac{\alpha}{1-\alpha}}(x-y)\rangle^{-\mathscr{L}}a(y)|dy
                          \\
                          &\lesssim
                          \langle k \rangle^{\frac{n\alpha}{1-\alpha}}\langle\langle k \rangle^{\frac{\alpha}{1-\alpha}}x\rangle^{-\mathscr{L}}\|a\|_{L^1}
                          \\
                          &\lesssim
                          \langle k \rangle^{\frac{n\alpha}{1-\alpha}}\langle\langle k \rangle^{\frac{\alpha}{1-\alpha}}x\rangle^{-\mathscr{L}}|Q|^{1-1/p}.
    \end{split}
  \end{equation}
  In addition, if $a$ satisfies the vanishing moment condition, we use the Taylor formula to deduce that
 \begin{equation}\label{for proof, 4}
    \begin{split}
     \widetilde{\Box}_k^{\alpha}a(x)&=\int_{\mathbb{R}^n}(\rho_k^{\alpha})^{\vee}(x-y)a(y)dy
                          \\
                          &=\int_{\mathbb{R}^n}\bigg((\rho_k^{\alpha})^{\vee}(x-y)-\sum_{|\gamma|\leqslant N}\frac{\partial^{\gamma}(\rho_k^{\alpha})^{\vee}(x)}{|\gamma|!}(-y)^{\gamma}\bigg)a(y)dy
                          \\
                          &=\int_{\mathbb{R}^n}\sum_{|\gamma|=N+1}\int_0^1\frac{(-y)^{\gamma}}{N!}(1-t)^N \partial^{\gamma}(\rho_k^{\alpha})^{\vee}(x-ty)a(y)dtdy.
    \end{split}
  \end{equation}
  Noticing that $(\rho_k^{\alpha})^{\vee}(x)=\langle k\rangle^{\frac{n\alpha}{1-\alpha}}\rho^{\vee}(\langle k\rangle^{\frac{\alpha}{1-\alpha}}x)e^{2\pi i \langle k\rangle^{\frac{\alpha}{1-\alpha}} k\cdot x}$,
  we obtain that
  \begin{equation}\label{for proof, 5}
    \begin{split}
      |(\partial^{\gamma}(\rho_k^{\alpha})^{\vee})(x-ty)|
                          \lesssim &
                          \langle k\rangle^{\frac{n\alpha}{1-\alpha}}
                          \sum_{\gamma_1+\gamma_2=\gamma}\langle k\rangle^{\frac{\alpha |\gamma_1|}{1-\alpha}}\langle k\rangle^{\frac{ |\gamma_2|}{1-\alpha}}
                          \langle \langle k\rangle^{\frac{\alpha}{1-\alpha}}(x-ty)\rangle^{-\mathscr{L}}
                          \\
                          \lesssim &
                          \langle k\rangle^{\frac{n\alpha}{1-\alpha}}\langle k\rangle^{\frac{ |\gamma|}{1-\alpha}}\langle \langle k\rangle^{\frac{\alpha}{1-\alpha}}x\rangle^{-\mathscr{L}}
    \end{split}
  \end{equation}
  for $x\in (Q^*)^c$, $y\in Q$, $t\in [0,1]$.
  Combining with (\ref{for proof, 4}) and (\ref{for proof, 5}), we deduce that
  \begin{equation*}
    \begin{split}
     |\widetilde{\Box}_k^{\alpha}a(x)|
     \lesssim &
     \int_{Q}\sum_{|\gamma|=N+1}\int_0^1\frac{|y|^{N+1}}{N!}(1-t)^N|(\partial^{\gamma}(\rho_k^{\alpha})^{\vee})(x-ty)|\cdot|a(y)|dtdy
     \\
     \lesssim &
     \langle k\rangle^{\frac{n\alpha}{1-\alpha}}\langle k\rangle^{\frac{N+1}{1-\alpha}}\langle \langle k\rangle^{\frac{\alpha}{1-\alpha}}x\rangle^{-\mathscr{L}}
     \int_{Q}|y|^{N+1}|a(y)|dy
     \\
     \lesssim &
     \langle k\rangle^{\frac{n\alpha}{1-\alpha}}\langle k\rangle^{\frac{N+1}{1-\alpha}}\langle \langle k\rangle^{\frac{\alpha}{1-\alpha}}x\rangle^{-\mathscr{L}} |Q|^{\frac{N+1}{n}}|Q|^{1-1/p}
     \\
     = &
     \langle k \rangle^{\frac{n\alpha}{1-\alpha}} (\sqrt[n]{|Q|}\langle k \rangle^{\frac{1}{1-\alpha}})^{N+1}
                                    \langle \langle k \rangle^{\frac{\alpha}{1-\alpha}} x\rangle^{-\mathscr{L}}|Q|^{1-1/p}
    \end{split}
  \end{equation*}
  for $x\in (Q^*)^c$.
\end{proof}
\begin{proposition}\label{proposition, p-infty local Hardy to mod.}
  Let $0<p\leqslant 1$, we have
  \begin{equation*}
    \|a\|_{M_{p,\infty}^{n(1-\alpha)(1-1/p),\alpha}}\lesssim 1
  \end{equation*}
  for all $h^p$-atoms.
\end{proposition}
\begin{proof}
Take $a$ to be an $h^p$ atom, and without loss of generality we may assume that $a$ is supported in a cube $Q$ centered at the origin.
We denote $Q^{\ast}$ the cube with side length $2\sqrt{n}l(Q)$ having the same center as $Q$, where the $l(Q)$ is the side length of $Q$.
Choose a smooth function $\rho$ with compact support, such that $\rho_{k}^{\alpha}\cdot \eta_k^{\alpha}=\eta_k^{\alpha}$ and $\Box_k^{\alpha}\circ \widetilde{\Box}_k^{\alpha}=\Box_k^{\alpha}$, where we denote
$\rho_k^{\alpha}(\xi)=\rho\left(\frac{\xi-\langle k\rangle^{\frac{\alpha}{1-\alpha}}k}{\langle k\rangle^{\frac{\alpha}{1-\alpha}}}\right)$ and $\widetilde{\Box}_k^{\alpha}=\rho_k^{\alpha}(D)$.
Using Lemma \ref{Young inequality}, we obtain
\begin{equation}
  \|\Box_k^{\alpha}f\|_{L^p}=\|\widetilde{\Box}_k^{\alpha}(\Box_k^{\alpha}f)\|_{L^p}
  \lesssim
  \langle k\rangle^{\frac{n\alpha(1/(p\wedge 1)-1)}{1-\alpha}}\|\mathscr{F}^{-1}\eta_k^{\alpha}\|_{L^{p\wedge 1}}\cdot \|\widetilde{\Box}_k^{\alpha}f\|_{L^p}
  \lesssim
  \|\widetilde{\Box}_k^{\alpha}f\|_{L^p}
\end{equation}
for any fixed $p\in (0,\infty]$.
Thus, we only need to verify
\begin{equation*}
  \|a\|_{\widetilde{M}_{p,\infty}^{n(1-\alpha)(1-1/p),\alpha}}
  := \sup_{k\in\mathbb{Z}^n} \langle k \rangle^{n(1-1/p)}\|\widetilde{\Box}_k^{\alpha}a\|_{L^p}
  \lesssim 1.
\end{equation*}

By the (quasi-)triangle inequality, we have
\begin{equation}\label{for proof, 6}
  \begin{split}
\|a\|_{\widetilde{M}_{p,\infty}^{n(1-\alpha)(1-1/p),\alpha}}
=&\sup_{k\in\mathbb{Z}^n} \langle k \rangle^{n(1-1/p)}\|\widetilde{\Box}_k^{\alpha}a\|_{L^p}
    \\
    \lesssim &
    \sup_{k\in\mathbb{Z}^n} \langle k \rangle^{n(1-1/p)}\|\widetilde{\Box}_k^{\alpha}a\chi_{Q^{\ast}}\|_{L^p}+\sup_{k\in\mathbb{Z}^n} \langle k \rangle^{n(1-1/p)}\|\widetilde{\Box}_k^{\alpha}a\chi_{(Q^{\ast})^c}\|_{L^p}
  \end{split}
\end{equation}
By the properties of $h^p$-atom, we obtain
\begin{equation*}
  \|\widetilde{\Box}_k^{\alpha}a\chi_{Q^{\ast}}\|_{L^p}
  \lesssim |Q|^{1/p}\|\widetilde{\Box}_k^{\alpha}a\|_{L^{\infty}}
  \lesssim |Q|^{1/p}\|a\|_{L^{\infty}}\lesssim 1.
\end{equation*}
Recalling $p\leqslant 1$, we deduce $\sup\limits_{k\in\mathbb{Z}^n} \langle k \rangle^{n(1-1/p)}\|\widetilde{\Box}_k^{\alpha}a\chi_{Q^{\ast}}\|_{L^p}\lesssim 1$.

We turn to the estimates of second term in (\ref{for proof, 6}).
Denote $N= [n(1/p-1)]$.
Using Proposition \ref{Proposition, pointwise estiamte for atoms}, we obtain that
\begin{equation}
  \begin{split}
    &\sup_{\substack{k\in\mathbb{Z}^n\\ \sqrt[n]{|Q|}\langle k \rangle^{\frac{1}{1-\alpha}}<1}}\langle k \rangle^{n(1-1/p)}\|\widetilde{\Box}_k^{\alpha}a\chi_{(Q^{\ast})^c}\|_{L^p}
    \\
    \lesssim &
    \sup_{\substack{k\in\mathbb{Z}^n\\ \sqrt[n]{|Q|}\langle k \rangle^{\frac{1}{1-\alpha}}<1}}\langle k \rangle^{n(1-1/p)}\|\langle k \rangle^{\frac{n\alpha}{1-\alpha}} (\sqrt[n]{|Q|}\langle k \rangle^{\frac{1}{1-\alpha}})^{N+1}
                                    \langle \langle k \rangle^{\frac{\alpha}{1-\alpha}} x\rangle^{-\mathscr{L}}|Q|^{1-1/p}\chi_{(Q^{\ast})^c}\|_{L^p}
    \\
    \lesssim &
    (\sqrt[n]{|Q|}\langle k \rangle^{\frac{1}{1-\alpha}})^{n(1-1/p)+N+1}\leqslant 1,
  \end{split}
\end{equation}
where we use $N+1+n(1-1/p)\geqslant 0$ and the fact that there is no $k\in \mathbb{Z}^n$ such that $\sqrt[n]{|Q|}\langle k \rangle^{\frac{1}{1-\alpha}}<1$ for big $h_p$-atom.
We also have
\begin{equation*}
  \begin{split}
    &\sup_{\substack{k\in\mathbb{Z}^n\\ \sqrt[n]{|Q|}\langle k \rangle^{\frac{1}{1-\alpha}}\geqslant 1}}
                               \langle k \rangle^{n(1-1/p)}\|\widetilde{\Box}_k^{\alpha}a\chi_{(Q^{\ast})^c}\|_{L^p}
     \\
    \lesssim &
    \sup_{\substack{k\in\mathbb{Z}^n\\ \sqrt[n]{|Q|}\langle k \rangle^{\frac{1}{1-\alpha}}\geqslant 1}}
                               \langle k \rangle^{n(1-1/p)}\|\langle k \rangle^{\frac{n\alpha}{1-\alpha}}\langle\langle k \rangle^{\frac{n\alpha}{1-\alpha}}x\rangle^{-\mathscr{L}}|Q|^{1-1/p}\chi_{(Q^{\ast})^c}\|_{L^p}
    \\
    \lesssim &
    \sup_{\substack{k\in\mathbb{Z}^n\\ \sqrt[n]{|Q|}\langle k \rangle^{\frac{1}{1-\alpha}}\geqslant 1}}(\sqrt[n]{|Q|}\langle k \rangle^{\frac{1}{1-\alpha}})^{n(1-1/p)}\leqslant 1,
  \end{split}
\end{equation*}
where we use $n(1-1/p)\leqslant 0$. Thus, we deduce that
\begin{equation*}
  \begin{split}
    &\sup_{k\in\mathbb{Z}^n} \langle k \rangle^{n(1-1/p)}\|\widetilde{\Box}_k^{\alpha}a\chi_{(Q^{\ast})^c}\|_{L^p}
    \\
    \lesssim &
    \sup_{\substack{k\in\mathbb{Z}^n\\ \sqrt[n]{|Q|}\langle k \rangle^{\frac{1}{1-\alpha}}<1}}
                               \langle k \rangle^{n(1-1/p)}\|\widetilde{\Box}_k^{\alpha}a\chi_{(Q^{\ast})^c}\|_{L^p}
    + \sup_{\substack{k\in\mathbb{Z}^n\\ \sqrt[n]{|Q|}\langle k \rangle^{\frac{1}{1-\alpha}}\geqslant 1}}
                               \langle k \rangle^{n(1-1/p)}\|\widetilde{\Box}_k^{\alpha}a\chi_{(Q^{\ast})^c}\|_{L^p}
    \lesssim
     1.
  \end{split}
\end{equation*}
\end{proof}

\begin{proposition}\label{proposition, p-p local Hardy to mod.}
  Let $0<p< 1$, we have
  \begin{equation*}
    \|a\|_{M_{p,p}^{n(1-\alpha)(1-2/p),\alpha}}\lesssim 1
  \end{equation*}
  for all $h^p$-atoms.
\end{proposition}
\begin{proof}
Take $a$ to be an $h^p$ atom, and without loss of generality we may assume that $a$ is supported in a cube $Q$ centered at the origin.
We denote $Q^{\ast}$ the cube with side length $2\sqrt{n}l(Q)$ having the same center as $Q$, where the $l(Q)$ is the side length of $Q$.
By the same argument as in the proof of Proposition \ref{proposition, p-infty local Hardy to mod.}, we only need to verify
\begin{equation*}
\|a\|_{\widetilde{M}_{p,p}^{n(1-\alpha)(1-2/p),\alpha}}
:=\| \{\langle k \rangle^{n(1-2/p)}\|\widetilde{\Box}_k^{\alpha}a\|_{L^p}\}_{k\in \mathbb{Z}^n}\|_{l^p}\lesssim 1.
\end{equation*}
By the (quasi-)triangle inequality, we have
\begin{equation*}
  \begin{split}
\|a\|_{\widetilde{M}_{p,p}^{n(1-\alpha)(1-2/p),\alpha}}
=&\| \{\langle k \rangle^{n(1-2/p)}\|\widetilde{\Box}_k^{\alpha}a\|_{L^p}\}_{k\in \mathbb{Z}^n}\|_{l^p}
    \\
    \lesssim &
    \| \{\langle k \rangle^{n(1-2/p)}\|\widetilde{\Box}_k^{\alpha}a\chi_{Q^*}\|_{L^p}\}_{k\in \mathbb{Z}^n}\|_{l^p}
    +
    \| \{\langle k \rangle^{n(1-2/p)}\|\widetilde{\Box}_k^{\alpha}a\chi_{(Q^*)^c}\|_{L^p}\}_{k\in \mathbb{Z}^n}\|_{l^p}
  \end{split}
\end{equation*}
Recalling $\|\widetilde{\Box}_k^{\alpha}a\chi_{Q^{\ast}}\|_{L^p}\lesssim 1$ obtained in the proof of Proposition \ref{proposition, p-infty local Hardy to mod.}, we deduce that
\begin{equation}
  \begin{split}
    \|\{\langle k \rangle^{n(1-2/p)}\|\widetilde{\Box}_k^{\alpha}a\chi_{Q^*}\|_{L^p}\}_{k\in \mathbb{Z}^n}\|_{l^p}
    \lesssim
    \|\{\langle k \rangle^{n(1-2/p)}\}_{k\in \mathbb{Z}^n}\|_{l^p}
    \lesssim 1,
  \end{split}
\end{equation}
where we use the fact $p<1$.
On the other hand, by the same method as in the proof of Proposition \ref{proposition, p-infty local Hardy to mod.}, we obtain that
\begin{equation}
  \begin{split}
    &\| \{\langle k \rangle^{n(1-2/p)}\|\widetilde{\Box}_k^{\alpha}a\chi_{(Q^*)^c}\|_{L^p}\}_{\sqrt[n]{|Q|}\langle k \rangle^{\frac{1}{1-\alpha}}<1}\|_{l_k^p}
    \\
    \lesssim &
    \| \{\langle k \rangle^{-n/p}(\sqrt[n]{|Q|}\langle k \rangle^{\frac{1}{1-\alpha}})^{n(1-1/p)+N+1}\}_{\sqrt[n]{|Q|}\langle k \rangle^{\frac{1}{1-\alpha}}<1}\|_{l_k^p}
    \\
    =&
    (\sqrt[n]{|Q|})^{n(1-1/p)+N+1}\}
    \left(\sum_{\sqrt[n]{|Q|}\langle k \rangle^{\frac{1}{1-\alpha}}<1}\langle k\rangle^{-n}(\langle k \rangle^{\frac{p}{1-\alpha}})^{n(1-1/p)+N+1}\right)^{1/p}
    \lesssim 1.
  \end{split}
\end{equation}
where we use the fact $N+1+n(1-1/p)>0$.
We also have
\begin{equation}
  \begin{split}
    &\| \{\langle k \rangle^{n(1-2/p)}\|\widetilde{\Box}_k^{\alpha}a\chi_{(Q^*)^c}\|_{L^p}\}_{\sqrt[n]{|Q|}\langle k \rangle^{\frac{1}{1-\alpha}}\geqslant 1}\|_{l_k^p}
    \\
    \lesssim &
    \| \{\langle k \rangle^{-n/p}(\sqrt[n]{|Q|}\langle k \rangle^{\frac{1}{1-\alpha}})^{n(1-1/p)}\}_{\sqrt[n]{|Q|}\langle k \rangle^{\frac{1}{1-\alpha}}\geqslant 1}\|_{l_k^p}
    \\
    =&
    (\sqrt[n]{|Q|})^{n(1-1/p)}\}
    \left(\sum_{\sqrt[n]{|Q|}\langle k \rangle^{\frac{1}{1-\alpha}}\geqslant 1}\langle k\rangle^{-n}(\langle k \rangle^{\frac{p}{1-\alpha}})^{n(1-1/p)}\right)^{1/p}
    \lesssim 1,
  \end{split}
\end{equation}
where we use the fact $p<1$.
Thus,
\begin{equation}
  \begin{split}
    \| \{\langle k \rangle^{n(1-2/p)}\|\widetilde{\Box}_k^{\alpha}a\chi_{(Q^*)^c}\|_{L^p}\}_{k\in \mathbb{Z}^n}\|_{l^p}
    \lesssim &
    \| \{\langle k \rangle^{n(1-2/p)}\|\widetilde{\Box}_k^{\alpha}a\chi_{(Q^*)^c}\|_{L^p}\}_{\sqrt[n]{|Q|}\langle k \rangle^{\frac{1}{1-\alpha}}<1}\|_{l_k^p}
    \\
    +&
    \| \{\langle k \rangle^{n(1-2/p)}\|\widetilde{\Box}_k^{\alpha}a\chi_{(Q^*)^c}\|_{L^p}\}_{\sqrt[n]{|Q|}\langle k \rangle^{\frac{1}{1-\alpha}}\geqslant 1}\|_{l_k^p}
    \lesssim 1.
  \end{split}
\end{equation}
\end{proof}

\section{Proof of Theorem \ref{theorem, sharpness of embedding from local Hardy spaces to alpha-modulation spaces}}
In this section, we give the proof of Theorem \ref{theorem, sharpness of embedding from local Hardy spaces to alpha-modulation spaces}, i.e. the sharp conditions of $h^{p_1}\subset M_{p_2,q_2}^{s_2,\alpha}$.

\subsection{Necessity of Theorem \ref{theorem, sharpness of embedding from local Hardy spaces to alpha-modulation spaces}}
\label{Necessity of Theorem local Hardy to mod.}
To verify the necessity, we first show that the embedding $h^{p_1}\subset M_{p_2,q_2}^{s_2,\alpha}$
actually implies the embedding about corresponding weighted sequences.

\begin{proposition}\label{Proposition, Necessity of local Hardy to mod.}
  Let $0<p_1<\infty$, $0<p_2, q_2\leqslant\infty$, $s\in\mathbb{R}$.
  If $h^{p_1}\subset M_{p_2,q_2}^{s_2,\alpha}$, then we have
  \begin{enumerate}
    \item $$1/p_2\leqslant 1/p_1,$$
    \item $$l_{p_1}^{n\left(1-1/p_1\right),1}\subset l_{q_2}^{n(1-\alpha)/q_2+n\alpha\left(1-1/p_2\right)+s_2,1},$$
    \item $$l_{p_1}^{n\alpha(1-1/p_1),\alpha}\subset l_{q_2}^{n\alpha(1-1/p_2)+s_2,\alpha}.$$
  \end{enumerate}
\end{proposition}
\begin{proof}
We only state the proof for $q_2<\infty$,
since the case of $q_2=\infty$ can be treated by the same method with a slight modification.
By the Littlewood-Paley characterization of local Hardy space ,
we actually have
\begin{equation}\label{for proof, 1}
  \|f\|_{M^{s_2,\alpha}_{p_2,q_2}}\lesssim \|f\|_{F_{p_1,2}^0}.
\end{equation}

Let $f\in\mathscr{S}$ be a nonzero function whose Fourier transform has compact support in $B(0,1)$.
Set $\widehat{f_{\lambda}}(\xi)=\widehat{f}(\xi/\lambda)$, $\lambda>0$.
Using (\ref{for proof, 1}) and the local property of $\alpha$-modulation and local hardy space,
we obtain $\|f_{\lambda}\|_{L^{p_2}}\lesssim\|f_{\lambda}\|_{L^{p_1}}$ for sufficiently small $\lambda$
and consequently ${\lambda}^{n(1-1/{p_2})}\lesssim{\lambda}^{n(1-1/{p_1})}$.
Letting $\lambda\rightarrow0^+$, we deduce
\begin{equation}\label{p condition of local hardy to mod.}
1/p_2\leqslant 1/p_1,
\end{equation}
which is the first desired condition.

Let $g$ be a nonzero Schwartz function whose Fourier transform has compact support in $\{\xi:3/4\leqslant|\xi|\leqslant 4/3\}$,
satisfying $g(\xi)=1$ on $\{\xi:7/8\leqslant|\xi|\leqslant 8/7\}$ .
Set $\widehat{g_j}(\xi):=\widehat{g}(\xi/2^j)$.
By the definition of $\Delta_j$, we have $\Delta_j g_j=g_j$ for $j\geqslant 0$, and $\Delta_l g_j=0$ if $l\neq j$.
Denote
\begin{equation}
  \widetilde{\Gamma_j}=\big\{k\in \mathbb{Z}^n: \text{supp} \eta_k^{\alpha}\subset \{\xi:(7/8)\cdot 2^j\leqslant|\xi|\leqslant (8/7)\cdot 2^j\}\big\},
\end{equation}
we have $|\widetilde{\Gamma_j}|\sim 2^{jn(1-\alpha)}$ for $j\geqslant N$, where $N$ is a sufficiently large number.
For a truncated (only finite nonzero items) nonnegative sequence $\vec{a}=\{a_j\}_{j\in \mathbb{N}}$,
we define
\begin{equation}
  G_N^h:=\sum\limits_{j\geqslant N} a_j g_j^h,\hspace{6mm}g_j^h(x):=g_j(x+jhe_0),
\end{equation}
where $h\in \mathbb{R}$, $e_0=(1,0,\cdots,0)$ is the unit vector of $\mathbb{R}^n$.
By the definition of $\alpha$-modulation space, we obtain that
\begin{equation}\label{test function for hardy to mod. left}
  \begin{split}
    &\|G_N^h\|_{M_{p_2,q_2}^{s_2,\alpha}}
      =\left( \sum_{k\in\mathbb{Z}^n}\langle k\rangle^{\frac{s_2q_2}{1-\alpha}}\|\Box_k^{\alpha}G_N^h\|_{L^{p_2}}^{q_2}\right)^{1/{q_2}}
      \geqslant \left(  \sum_{j\in \mathbb{N}_0} \sum_{k\in \widetilde{\Gamma_j}}
      \langle k \rangle^{\frac{s_2q_2}{1-\alpha}}\|\Box_k^{\alpha} G_N^h\|_{L^{p_2}}^{q_2}\right)^{1/{q_2}}
      \\
      &
      =\left(  \sum_{j=N}^{\infty} \sum_{k\in \widetilde{\Gamma_j}} a_j^{q_2} \|\mathscr{F}^{-1}\eta_k^{\alpha}\|_{L^{p_2}}^{q_2}
        \langle k \rangle^{\frac{s_2q_2}{1-\alpha}} \right)^{1/{q_2}}
      \sim\left(  \sum_{j=N}^{\infty} \sum_{k\in \widetilde{\Gamma_j}}  a_j^{q_2} \langle k \rangle^{\frac{n\alpha q_2}{1-\alpha}(1-1/p_2)}
        2^{js_2q_2} \right)^{1/{q_2}}
      \\&
      \sim \left( \sum_{j=N}^{\infty} \sum_{k\in \widetilde{\Gamma_j}} a_j^{q_2}
        2^{j{q_2}(n\alpha(1-1/p_2)+s_2)} \right)^{1/{q_2}}
      \sim\left( \sum_{j=N}^{\infty} |\widetilde{\Gamma_j}|a_j^{q_2}
      2^{jq_2(n\alpha(1-1/p_2)+s_2)} \right)^{1/{q_2}}
      \\&
      =\left( \sum_{j=N}^{\infty} a_j^{q_2}
      2^{jq_2(n\alpha(1-1/p_2)+s_2)}2^{jn(1-\alpha)} \right)^{1/{q_2}}
      =\|\{a_j\}_{j\geqslant N}\|_{l_{q_2}^{n(1-\alpha)/q_2+n\alpha\left(1-1/p_2\right)+s_2,1}}.
  \end{split}
\end{equation}

On the other hand,  using the orthogonality of $\{g_j^h\}$ as $h\rightarrow \infty$, we deduce that
\begin{equation}\label{test function for hardy to mod. right}
  \begin{split}
    \|G_N^h\|_{F_{p_1,2}^0}&=\left\| \left( \sum_{j\in \mathbb{N}_0} {|\Delta_j G_N^h|}^2 \right)^{1/2}\right\|_{L^{p_1}}
                            =\left\| \left( \sum_{j=N}^{\infty} |a_j g_j^h|^2 \right)^{1/2} \right\|_{L^{p_1}}
                          \\
                           &=\left( \int_{\mathbb{R}^n}\left( \sum_{j=N}^{\infty} |a_j g_j^h|^2 \right)^{{p_1}/{2}} dx \right)^{1/{p_1}}
                            \xlongrightarrow[]{h\rightarrow\infty}
                            \left( \int_{\mathbb{R}^n} \sum_{j=N}^{\infty} |a_j g_j|^{p_1}  dx \right)^{1/{p_1}}
                          \\
                           &\simeq \left( \sum_{j=N}^{\infty} a_j^{p_1}2^{jn(1-1/{p_1})p_1} \right)^{1/{p_1}}
                            =\|\{a_j\}_{j\geqslant N}\|_{l_{p_1}^{n\left(1-1/p_1\right),1}}.
  \end{split}
\end{equation}

Combining (\ref{test function for hardy to mod. right}) and (\ref{test function for hardy to mod. left}), we have
$$
\|\{a_j\}_{j\geqslant N}\|_{l_{q_2}^{n(1-\alpha)/q_2+n\alpha\left(1-1/p_2\right)+s_2,1}}
\lesssim
\|\{a_j\}_{j\geqslant N}\|_{l_{p_1}^{n\left(1-1/p_1\right),1}},
$$
which implies the desired embedding
$l_{p_1}^{n\left(1-1/p_1\right),1}\subset l_{q_2}^{n(1-\alpha)/q_2+n\alpha\left(1-1/p_2\right)+s_2,1}$.

Next, we turn to the proof of $l_{p_1}^{n\alpha(1-1/p_1),\alpha}\subset l_{q_2}^{n\alpha(1-1/p_2)+s_2,\alpha}.$
Let $f\in\mathscr{S}$ be a nonzero smooth function whose Fourier transform has small support, such that
$\Box_k^{\alpha}f_k=f_k$ and $\Box_l^{\alpha}f_k=0$ if $k\neq l$,
where we denote $\widehat{f_k}(x)=\widehat{f}\left(\frac{\xi-\langle k\rangle^{\frac{\alpha}{1-\alpha}} k}{\langle k\rangle^{\frac{\alpha}{1-\alpha}}}\right)$.
For a truncated (only finite nonzero items) nonnegative sequence $\vec{b}=\{b_k\}_{k\in \mathbb{Z}^n}$,
we define
\begin{equation}
  F^h(x)=\sum\limits_{k\in \mathbb{Z}^n}b_kf_k^h,\hspace{6mm}f_k^h(x)=f_k(x-kh),
\end{equation}
where $h\in \mathbb{R}$.
By a direct computation, we have
\begin{equation}\label{for proof, 2}
  \begin{split}
    \|F^h\|_{M_{p_2,q_2}^{s_2,\alpha}}
    = &
    \left( \sum_{k\in \mathbb{Z}^n}b_k^{q_2}\langle k\rangle^{\frac{s_2q_2}{1-\alpha}}\|f_k\|_{L^{p_2}}^{q_2} \right)^{1/q_2}
    \\
    \sim &
    \left( \sum_{k\in \mathbb{Z}^n}b_k^{q_2}\langle k\rangle^{\frac{s_2q_2}{1-\alpha}}
    \langle k\rangle^{\frac{n\alpha}{1-\alpha}(1-1/p_2)q_2} \right)^{1/q_2}
    \sim
    \|\vec{b}\|_{l_{q_2}^{n\alpha(1-1/p_2)+s_2,\alpha}}.
  \end{split}
\end{equation}
On the other hand,
for $p_1>1$, using the orthogonality of $\{f_k^h\}_{k\in \mathbb{Z}^n}$ as $h\rightarrow \infty$,
we obtain
\begin{equation*}
  \begin{split}
    \|F^h\|_{h^{p_1}}&\sim \|F^h\|_{L^{p_1}}
                       =    \left( \int_{\mathbb{R}^n}\left|\sum_{k\in \mathbb{Z}^n } b_k f_k^h\right|^{p_1}dx \right)^{\frac{1}{p_1}}
                       \xlongrightarrow[]{h\rightarrow\infty}
                            \left( \int_{\mathbb{R}^n}\sum_{k\in \mathbb{Z}^n } \left|b_k f_k \right|^{p_1}dx \right)^{\frac{1}{p_1}}
                       \\
                       &\sim \left( \sum_{k\in \mathbb{Z}^n}b_k^{p_1}\langle k\rangle^{\frac{n\alpha}{1-\alpha}(1-1/p_1)p_1} \right)^{1/p_1}
                       \sim \|\vec{b}\|_{l_{p_1}^{n\alpha(1-1/p_1),\alpha}}.
  \end{split}
\end{equation*}
For $p_1\leqslant1$, we use quasi-triangle inequality to deduce that
\begin{equation*}
  \begin{split}
    \|F^h\|_{h^{p_1}}^{p_1}&=         \|\sum_{k\in \mathbb{Z}^n}b_k f_k^h\|_{h^{p_1}}^{p_1}
                              \leqslant \sum_{k\in \mathbb{Z}^n}\|b_k f_k^h\|_{h^{p_1}}^{p_1}
                              =
                                        \sum_{k\in \mathbb{Z}^n}\|b_k f_k\|_{L^{p_1}}^{p_1}
                              \\
                              &\sim     \sum_{k\in \mathbb{Z}^n}b_k^{p_1}\langle k\rangle^{\frac{n\alpha}{1-\alpha}(1-1/p_1)p_1}
                              \sim      \|\vec{b}\|_{l_{p_1}^{n\alpha(1-1/p_1),\alpha}}^{p_1},
  \end{split}
\end{equation*}
where we use the fact $\|f_k\|_{L^{p_1}}\sim \|f_k\|_{L^{p_1}}$ (see Proposition \ref{equivlent of Lp and hp with Fourier support}).
So we have
\begin{equation}\label{for proof, 3}
  \lim_{h\rightarrow \infty}\|F^h\|_{h^{p_1}}\lesssim  \|\vec{b}\|_{l_{p_1}^{n\alpha(1-1/p_1),\alpha}}^{p_1}.
\end{equation}
Combining (\ref{for proof, 2}) and (\ref{for proof, 3}), we obtain
\begin{equation}
  \|\vec{b}\|_{l_{q_2}^{n\alpha(1-1/p_2)+s_2,\alpha}} \lesssim \|\vec{b}\|_{l_{p_1}^{n\alpha(1-1/p_1),\alpha}}
\end{equation}
which implies the desired embedding
$l_{p_1}^{n\alpha(1-1/p_1),\alpha}\subset  l_{q_2}^{n\alpha(1-1/p_2)+s_2,\alpha}.$
\end{proof}

Now, we are in the position to verify the necessity for $h^{p_1}\subset M_{p_2,q_2}^{s_2,\alpha}$.
By Proposition \ref{Proposition, Necessity of local Hardy to mod.}, we obtain $1/p_2\leqslant 1/p_1$ and
\begin{equation*}
  l_{p_1}^{n\left(1-1/p_1\right),1}\subset l_{q_2}^{n(1-\alpha)/q_2+n\alpha\left(1-1/p_2\right)+s_2,1},
  \hspace{6mm}
  l_{p_1}^{n\alpha(1-1/p_1),\alpha}\subset l_{q_2}^{n\alpha(1-1/p_2)+s_2,\alpha}.
\end{equation*}
By Lemma \ref{lemma, Sharpness of embedding, for alpha decomposition},
$l_{p_1}^{n\left(1-1/p_1\right),1}\subset l_{q_2}^{n(1-\alpha)/q_2+n\alpha\left(1-1/p_2\right)+s_2,1}$
implies $s_2\leqslant n\alpha(1/p_2-1/p_1)+n(1-\alpha)(1-1/p_1-1/q_2)$ and the inequality is strict if $1/q_2>1/p_1$.
On the other hand, by Lemma \ref{lemma, Sharpness of embedding, for alpha decomposition} and $l_{p_1}^{n\alpha(1-1/p_1),\alpha}\subset l_{q_2}^{n\alpha(1-1/p_2)+s_2,\alpha}$,
we obtain $s_2\leqslant n\alpha(1/p_2-1/p_1)$ for $1/q_2\leqslant 1/p_1$,
and $s_2< n(1-\alpha)(1/p_1-1/q_2)+n\alpha(1/p_2-1/p_1)$ for $1/q_2>1/p_1$.

Combining with the above estimates, we obtain
$s_2\leqslant n\alpha(1/p_2-1/p_1)+(1-\alpha)A(p_1,q_2)$, with strict inequality for $1/q_2> 1/p_1$.

\subsection{Sufficiency of Theorem \ref{theorem, sharpness of embedding from local Hardy spaces to alpha-modulation spaces}}
We only need to verify $h^{p_1}\subset M_{p_1,q_2}^{(1-\alpha)A(p_1,q_2),\alpha}$ for $1/q_2\leqslant 1/p_1$,
and $h^{p_1}\subset M_{p_1,q_2}^{(1-\alpha)A(p_1,q_2)-\epsilon,\alpha}$ for $1/q_2> 1/p_1$, where $\epsilon$ is any positive number.
Then the final conclusion follows by $M_{p_1,q_2}^{(1-\alpha)A(p_1,q_2),\alpha}\subset M_{p_2,q_2}^{n\alpha(1/p_2-1/p_1)+(1-\alpha)A(p_1,q_2),\alpha}$ or
$M_{p_1,q_2}^{(1-\alpha)A(p_1,q_2)+\epsilon,\alpha}\subset M_{p_2,q_2}^{n\alpha(1/p_2-1/p_1)+(1-\alpha)A(p_1,q_2)+\epsilon,\alpha}$.

\textbf{For $1/q_2\leqslant 1/p_1$.}
We want to verify that
\begin{equation*}
  h^{p_1}\subset M_{p_1,q_2}^{(1-\alpha)A(p_1,q_2),\alpha}.
\end{equation*}
Taking a fixed $f\in h^{p_1} (p_1\leqslant 1)$, we can find a collection of $h_{p_1}$-atoms $\{a_j\}_{j=1}^{\infty}$ and a sequence of complex numbers $\{\lambda_j\}_{j=1}^{\infty}$
which is depend on $f$,
such that $f=\sum_{j=1}^{\infty}\lambda_ja_j$ and
\begin{equation}
  \left(\sum_{j=1}^{\infty}|\lambda|^p\right)^{1/p}\leqslant C\|f\|_{h_p},
\end{equation}
where the constant $C$ is independent of $f$.
For $p_1\leqslant 1$, we use Proposition \ref{proposition, p-infty local Hardy to mod.} to deduce that
\begin{equation}\label{for proof, 7}
  \begin{split}
    \|f\|_{M_{p_1,\infty}^{(1-\alpha)A(p_1,\infty),\alpha}}
    =
    \|f\|_{M_{p_1,\infty}^{n(1-\alpha)(1-1/p_1),\alpha}}
    = &
    \|\sum_{j=1}^{\infty}\lambda_ja_j\|_{M_{p_1,\infty}^{n(1-\alpha)(1-1/p_1),\alpha}}
    \\
    \lesssim &
    \left(\sum_{j=1}^{\infty}\|\lambda_ja_j\|^{p_1}_{M_{p_1,\infty}^{n(1-\alpha)(1-1/p_1),\alpha}}\right)^{1/p_1}
    \\
    \lesssim &
    \left(\sum_{j=1}^{\infty}|\lambda_j|^{p_1}\right)^{1/p_1}\lesssim \|f\|_{h_{p_1}}.
  \end{split}
\end{equation}
Similarly, we use Proposition \ref{proposition, p-p local Hardy to mod.} to deduce that
\begin{equation}\label{for proof, 8}
  \|f\|_{M_{p_1,p_1}^{(1-\alpha)A(p_1,p_1),\alpha}}=\|f\|_{M_{p_1,p_1}^{n(1-\alpha)(1-2/p_1),\alpha}}\lesssim \|f\|_{h^{p_1}}
\end{equation}
for $p_1<1$.
By a direct calculation, we also have
\begin{equation}\label{for proof, 9}
  \|f\|_{M_{\infty,\infty}^{(1-\alpha)A(\infty,\infty),\alpha}}=\|f\|_{M_{\infty,\infty}^{0,\alpha}}
  =
  \sup_{k\in \mathbb{Z}^n}\|\Box_k^{\alpha}f\|_{L^{\infty}}\lesssim \|f\|_{L^{\infty}}.
\end{equation}
Recalling $\|f\|_{M_{2,2}^{(1-\alpha)A(2,2),\alpha}}=\|f\|_{M_{2,2}^{0,\alpha}}\sim \|f\|_{L^2}$, we obtain the following inequality:
\begin{equation}\label{for proof, 10}
  \|f\|_{M_{2,2}^{(1-\alpha)A(2,2),\alpha}}\lesssim \|f\|_{L^2}.
\end{equation}
By an interpolation among (\ref{for proof, 7}),(\ref{for proof, 8}),(\ref{for proof, 9}) and (\ref{for proof, 10}), we obtain the desired conclusion.

\textbf{For $1/q_2> 1/p_1$.}
We want to verify that
\begin{equation*}
  h^{p_1}\subset M_{p_1,q_2}^{(1-\alpha)A(p_1,q_2)-\epsilon,\alpha},
\end{equation*}
where $\epsilon$ is any fixed positive number.
In fact, observing that
\begin{equation*}
  ((1-\alpha)A(p_1,q_2)-\epsilon)/n+(1-\alpha)/q_2<((1-\alpha)A(p_1,p_1)/n+(1-\alpha)/p_1
\end{equation*}
for $1/q_2>1/p_1$,
we use Lemma \ref{Lemma, embedding of alpha modulation spaces} to deduce
\begin{equation*}
  M_{p_1,p_1}^{(1-\alpha)A(p_1,p_1),\alpha}\subset M_{p_1,q_2}^{(1-\alpha)A(p_1,q_2)-\epsilon,\alpha}.
\end{equation*}
Recalling $h^{p_1}\subset M_{p_1,p_1}^{(1-\alpha)A(p_1,p_1),\alpha}$, we obtain the desired conclusion:
\begin{equation*}
  h^{p_1}\subset M_{p_1,p_1}^{(1-\alpha)A(p_1,p_1),\alpha}\subset M_{p_1,q_2}^{(1-\alpha)A(p_1,q_2)-\epsilon,\alpha}
\end{equation*}
for any fixed positive number.

\section{Proof of Theorem \ref{theorem, sharpness of embedding from alpha-modulation spaces to local Hardy spaces}}\label{proof of Th. alpha mod. to local Hardy}
\subsection{Necessity of Theorem \ref{theorem, sharpness of embedding from alpha-modulation spaces to local Hardy spaces}}
As in the proof of Theorem \ref{theorem, sharpness of embedding from local Hardy spaces to alpha-modulation spaces},
we use following proposition to show that the necessity of $M_{p_2,q_2}^{s_2,\alpha}\subset h^{p_1}$ can be reduced to
the necessity of embedding between corresponding weighted sequences.

\begin{proposition}\label{Proposition, Necessity of mod to local Hardy.}
  Let $0<p_2<\infty$, $0<p_1, q_1\leqslant\infty$, $s_1\in\mathbb{R}$.
  If $M_{p_1,q_1}^{s_1,\alpha} \subset h^{p_2}$, then we have
  \begin{enumerate}
    \item $$1/p_2\leqslant 1/p_1,$$
    \item $$l_{q_1}^{n(1-\alpha)/q_1+n\alpha\left(1-1/p_1\right)+s_1,1}\subset l_{p_2}^{n\left(1-1/p_2\right),1},$$
    \item $$l_{q_1}^{n\alpha(1-1/p_1)+s_1,\alpha}\subset l_{p_2}^{n\alpha(1-1/p_2),\alpha}.$$
  \end{enumerate}
\end{proposition}
\begin{proof}
Take $f\in\mathscr{S}$ to be a nonzero smooth function whose Fourier transform has compact support in $B(0,1)$.
Denote $\widehat{f_{\lambda}}(\xi)=\widehat{f}(\xi/\lambda)$, $\lambda>0$.
By the assumption, we have
\begin{equation}
  \|f_{\lambda}\|_{F_{p_2,2}^0}\lesssim \|f_{\lambda}\|_{M^{s_1,\alpha}_{p_1,q_1}}
\end{equation}
for $\lambda\leqslant 1$.
By the local property of $\alpha$-modulation and Triebel space, we actually have
$$\|f_{\lambda}\|_{L^{p_2}}\lesssim\|f_{\lambda}\|_{L^{p_1}},$$
and consequently ${\lambda}^{n(1-1/{p_2})}\lesssim{\lambda}^{n(1-1/{p_1})}$.
Letting $\lambda\rightarrow0^+$, we conclude
\begin{equation}\label{p condition of mod. to local hardy}
1/p_2\leqslant1/p_1.
\end{equation}
Let $g$ be a nonzero Schwartz function whose Fourier transform has compact support in $\{\xi:3/4\leqslant|\xi|\leqslant 4/3\}$.
Set $\widehat{g_j}(\xi):=\widehat{g}(\xi/2^j)$.
By the definition of $\Delta_j$, we have $\Delta_j g_j=g_j$ for $j\geqslant 0$, and $\Delta_l g_j=0$ if $l\neq j$.
Denote
\begin{equation}
  \Gamma_j=\big\{k\in \mathbb{Z}^n: \text{supp} \eta_k^{\alpha}\cap \{\xi:(3/4)\cdot 2^j\leqslant|\xi|\leqslant (4/3)\cdot 2^j\}\neq \emptyset \big\}.
\end{equation}
We have $|\Gamma_j|\sim 2^{jn(1-\alpha)}$.
For a truncated (only finite nonzero items) nonnegative sequence $\vec{a}=\{a_j\}_{j\in \mathbb{N}_0}$,
we define
\begin{equation}
  G^h:=\sum\limits_{j\in \mathbb{N}_0} a_j g_j^h,\hspace{6mm}g_j^h(x):=g_j(x+jhe_0),
\end{equation}
where $h\in \mathbb{R}$, $e_0=(1,0,\cdots,0)$ is the unit vector in $\mathbb{R}^n$.
By the definition of $\alpha$-modulation space, we obtain that
\begin{equation}\label{test function for mod to hardy. right}
  \begin{split}
    &\|G^h\|_{M_{p_1,q_1}^{s_1,\alpha}}
      =\left( \sum_{k\in\mathbb{Z}^n}\langle k\rangle^{\frac{s_1q_1}{1-\alpha}}\|\Box_k^{\alpha}G^h\|_{L^{p_1}}^{q_1}\right)^{\frac{1}{q_1}}
      \lesssim \left(  \sum_{j\in \mathbb{N}_0} \sum_{k\in \Gamma_j}
      \langle k \rangle^{\frac{s_1q_1}{1-\alpha}}\|a_j\Box_k^{\alpha} g_j^h\|_{L^{p_1}}^{q_1}\right)^{\frac{1}{q_1}}
      \\
      &
      \lesssim\left(  \sum_{j\in \mathbb{N}_0} \sum_{k\in \Gamma_j} a_j^{q_1} \|\mathscr{F}^{-1}\eta_k^{\alpha}\|_{L^{p_1}}^{q_1}
        \langle k \rangle^{\frac{s_1q_1}{1-\alpha}} \right)^{\frac{1}{q_1}}
      \sim\left(  \sum_{j\in \mathbb{N}_0} \sum_{k\in \Gamma_j}  a_j^{q_1} \langle k \rangle^{\frac{n\alpha q_1}{1-\alpha}(1-1/p_1)}
        2^{js_1q_1} \right)^{\frac{1}{q_1}}
      \\&
      \sim \left( \sum_{j\in \mathbb{N}_0} \sum_{k\in \Gamma_j} a_j^{q_1}
        2^{j{q_1}(n\alpha(1-{1}/{p_1})+s_1)} \right)^{\frac{1}{q_1}}
      \sim\left( \sum_{j\in \mathbb{N}_0} |\Gamma_j|a_j^{q_1}
      2^{jq_1(n\alpha(1-1/p_1)+s_1)} \right)^{\frac{1}{q_1}}
      \\&
      =\left( \sum_{j\in \mathbb{N}_0} a_j^{q_1}
      2^{jq_1(n\alpha(1-1/p_1)+s_1)}2^{jn(1-\alpha)} \right)^{\frac{1}{q_1}}
      =\|\vec{a}\|_{l_{q_1}^{n(1-\alpha)/q_1+n\alpha\left(1-1/p_1\right)+s_1,1}}.
  \end{split}
\end{equation}

On the other hand,  by the same argument as in the proof of Proposition \ref{Proposition, Necessity of local Hardy to mod.}, we deduce that
\begin{equation}
    \lim_{h\rightarrow \infty}\|G^h\|_{L^{p_2}}=\|\vec{a}\|_{l_{p_2}^{n\left(1-1/p_2\right),1}}.
\end{equation}
Since $G^h$ is a Schwartz function, by the definition of local Hardy space, we have $\|G^h\|_{L^{p_2}}\lesssim \|G^h\|_{h^{p_2}}$. Thus,
\begin{equation}\label{test function for mod to hardy. left}
  \|\vec{a}\|_{l_{p_2}^{n\left(1-1/p_2\right),1}}=\lim_{h\rightarrow \infty}\|G^h\|_{L^{p_2}}\lesssim \lim_{h\rightarrow \infty}\|G^h\|_{h^{p_2}}.
\end{equation}

Combining (\ref{test function for mod to hardy. right}) and (\ref{test function for mod to hardy. left}), we have
$$
\|\vec{a}\|_{l_{p_2}^{n\left(1-1/p_2\right),1}}
\lesssim
\|\vec{a}\|_{l_{q_1}^{n(1-\alpha)/q_1+n\alpha\left(1-1/p_1\right)+s_1,1}},
$$
which implies the desired embedding
$l_{q_1}^{n(1-\alpha)/q_1+n\alpha\left(1-1/p_1\right)+s_1,1}\subset l_{p_2}^{n\left(1-1/p_2\right),1}.$

Next, we turn to verify $l_{q_1}^{n\alpha(1-1/p_1)+s_1,\alpha}\subset l_{p_2}^{n\alpha(1-1/p_2),\alpha}.$
Let $f\in\mathscr{S}$ be a nonzero smooth function whose Fourier transform has small support, such that
$\Box_k^{\alpha}f_k=f_k$ and $\Box_l^{\alpha}f_k=0$ if $k\neq l$,
where we denote $\widehat{f_k}(x)=\widehat{f}\left(\frac{\xi-\langle k\rangle^{\frac{\alpha}{1-\alpha}} k}{\langle k\rangle^{\frac{\alpha}{1-\alpha}}}\right)$.
For a truncated (only finite nonzero items) nonnegative sequence $\vec{b}=\{b_k\}_{k\in \mathbb{Z}^n}$,
we define
\begin{equation}
  F^h(x)=\sum\limits_{k\in \mathbb{Z}^n}b_kf_k^h,\hspace{6mm}f_k^h(x)=f_k(x-kh),
\end{equation}
where $h\in \mathbb{R}$.
Observing that $F^h$ is a Schwartz function and $\|F^h\|_{L^{p_2}}\leqslant \|F^h\|_{h^{p_2}}$,
we use the assumption to deduce
\begin{equation}
  \|F^h\|_{L^{p_2}}\lesssim \|F^h\|_{M_{p_1,q_1}^{s_1,\alpha}}.
\end{equation}
By the same argument as in the proof of Proposition \ref{Proposition, Necessity of local Hardy to mod.}, we obtain
\begin{equation*}
    \lim_{h\rightarrow \infty}\|F^h\|_{L^{p_2}} \sim \|\vec{b}\|_{l_{p_2}^{n\alpha(1-1/p_2),\alpha}},
    \hspace{6mm}
    \|F^h\|_{M_{p_1,q_1}^{s_1,\alpha}}
         \sim\|\vec{b}\|_{l_{q_1}^{n\alpha(1-1/p_1)+s_1,\alpha}},
\end{equation*}
which implies
\begin{equation}
  \|\vec{b}\|_{l_{p_2}^{n\alpha(1-1/p_2),\alpha}}\lesssim\|\vec{b}\|_{l_{q_1}^{n\alpha(1-1/p_1)+s_1,\alpha}}.
\end{equation}
By the arbitrary of $\vec{b}$, we obtain the desired conclusion.
\end{proof}

Now, we turn to verify the necessity of $M_{p_1,q_1}^{s_1,\alpha}\subset h^{p_2}$.
Using Proposition \ref{Proposition, Necessity of mod to local Hardy.}, we obtain $1/p_2\leqslant 1/p_1$ and the embedding relations:
\begin{equation*}
  l_{q_1}^{n(1-\alpha)/q_1+n\alpha\left(1-1/p_1\right)+s_1,1}\subset l_{p_2}^{n\left(1-1/p_2\right),1},
  \hspace{6mm}
  l_{q_1}^{n\alpha(1-1/p_1)+s_1,\alpha}\subset l_{p_2}^{n\alpha(1-1/p_2),\alpha}.
\end{equation*}
By Lemma \ref{lemma, Sharpness of embedding, for alpha decomposition},
$l_{q_1}^{n(1-\alpha)/q_1+n\alpha\left(1-1/p_1\right)+s_1,1}\subset l_{p_2}^{n\left(1-1/p_2\right),1}$
implies $s_1\geqslant n\alpha(1/p_1-1/p_2)+n(1-\alpha)(1-1/p_2-1/q_1)$ and the inequality is strict if $1/p_2>1/q_1$.
On the other hand, by Lemma \ref{lemma, Sharpness of embedding, for alpha decomposition} and $l_{q_1}^{n\alpha(1-1/p_1)+s_1,\alpha}\subset l_{p_2}^{n\alpha(1-1/p_2),\alpha}$,
we obtain $s_1\geqslant n\alpha(1/p_1-1/p_2)$ for $1/p_2\leqslant 1/q_1$,
and $s_1> n\alpha(1/p_1-1/p_2)+n(1-\alpha)(1/p_2-1/q_1)$ for $1/p_2>1/q_1$.

Combining with the above estimates, we obtain
$s_1\geqslant n\alpha(1/p_1-1/p_2)+(1-\alpha)B(p_2,q_1)$, with strict inequality for $1/p_2> 1/q_1$.

\subsection{Sufficiency of Theorem \ref{theorem, sharpness of embedding from alpha-modulation spaces to local Hardy spaces}}
As in the proof of Theorem \ref{theorem, sharpness of embedding from local Hardy spaces to alpha-modulation spaces},
we actually only need to verify $M_{p_2,q_1}^{(1-\alpha)B(p_2,q_1),\alpha}\subset h^{p_2}$ for $1/p_2\leqslant 1/q_1$,
and $M_{p_2,q_1}^{(1-\alpha)B(p_2,q_1)+\epsilon,\alpha}\subset h^{p_2}$ for $1/p_2> 1/q_1$, where $\epsilon$ is any positive number.

\textbf{For $1/p_2\leqslant 1/q_1.$} We want to verify
\begin{equation}\label{for proof, 11}
  M_{p_2,q_1}^{(1-\alpha)B(p_2,q_1),\alpha}\subset h^{p_2}.
\end{equation}
Observing that $B(p_2,q_1)=0$ and $M_{p_2,q_1}^{0,\alpha}\subset M_{p_2,\widetilde{q_1}}^{0,\alpha}$ for $1/q_1\geqslant (1-1/p_2)\vee 1/p_2$,
where $1/\widetilde{q_1}=(1-1/p_2)\vee 1/p_2$, we only need to prove (\ref{for proof, 11}) for $1/p_2\leqslant 1/q_1\leqslant 1-1/p_2$ and $p_2=q_1\leqslant 2$.

By a dual argument, in $1/p_2\leqslant 1/q_1\leqslant 1-1/p_2$, (\ref{for proof, 11}) can be verified by the sufficiency of Theorem \ref{theorem, sharpness of embedding from local Hardy spaces to alpha-modulation spaces}
proved before. Thus, we only need to verify (\ref{for proof, 11}) for $p_2=q_1\leqslant 2$.
Observing that (\ref{for proof, 11}) has been verified for $p_2=q_1=2$, by an interpolation argument, we only need to show (\ref{for proof, 11}) for $p_2=q_1\leqslant 1$.

In fact, when $p_2=q_1\leqslant 1$, we use the quasi-triangle inequality to deduce that
\begin{equation}\label{for proof, 12}
  \|f\|_{h^{p_2}}=\|\sum_{k\in \mathbb{Z}^n}\Box_k^{\alpha}f\|_{h^{p_2}}
  \leqslant
  \left(\sum_{k\in \mathbb{Z}^n}\|\Box_k^{\alpha}f\|^{p_2}_{h^{p_2}}\right)^{1/p_2}
  \sim
  \left(\sum_{k\in \mathbb{Z}^n}\|\Box_k^{\alpha}f\|^{p_2}_{L^{p_2}}\right)^{1/p_2}
  =
  \|f\|_{M_{p_2,q_1}^{0,\alpha}},
\end{equation}
which is just the embedding (\ref{for proof, 11}) for $p_2=q_1\leqslant 1$.

\textbf{For $1/p_2> 1/q_1.$} We want to verify
\begin{equation*}
  M_{p_2,q_1}^{(1-\alpha)B(p_2,q_1)+\epsilon,\alpha}\subset h^{p_2}
\end{equation*}
for any fixed positive number $\epsilon$.
In fact, by the embedding $M_{p_2,p_2}^{(1-\alpha)B(p_2,p_2),\alpha}\subset h^{p_2}$ proved before,
observing that
\begin{equation*}
  n(1-\alpha)/p_2+(1-\alpha)B(p_2,p_2)<n(1-\alpha)/q_1+(1-\alpha)B(p_2,q_1)+\epsilon,
\end{equation*}
we use Lemma \ref{Lemma, embedding of alpha modulation spaces} to deduce
\begin{equation*}
  M_{p_2,q_1}^{(1-\alpha)B(p_2,q_1)+\epsilon,\alpha}\subset M_{p_2,p_2}^{(1-\alpha)B(p_2,p_2),\alpha}\subset h^{p_2},
\end{equation*}
which is the desired conclusion.

\section{Embedding between $L^1$ and $\alpha$-modulation spaces}

\subsection{Proof of Theorem \ref{theorem, sharpness of embedding from L1 to alpha-modulation spaces}.}
We first give the proof for necessity.
It is known that $h^p=L^p$ when $p>1$, but $h^1\subsetneqq   L^1$, due to the different structure between
$h^1$ and $L^1$, we have the a additional restriction associated with $L^1\subset M_{p,q}^{s,\alpha}$.
\begin{proposition}\label{proposition, necessity of L1 to mod}
  Let $0<p\leqslant \infty$ $0<q<\infty$, $s\in \mathbb{R}$, then $L^1\subset M_{p,q}^{s,\alpha}$ implies
  \begin{equation*}
    s< n\alpha(1/p-1)+n(1-\alpha)A(1,q)
  \end{equation*}
\end{proposition}
\begin{proof}
  Take $f$ to be a smooth function satisfying that $\text{supp} \widehat{f}\subset B(0,2)$ and $\widehat{f}(\xi)=1$ on $B(0,1)$.
  Denote $\widehat{f_j}(\xi)=f(\xi/2^j)$ for $j\geqslant 1$. By the scaling of $L^1$, we obtain $\|f_j\|_{L^1}\sim 1$ for all $j$.
  However, by the assumption of $f_j$,
  for any finite subset of $\mathbb{Z}^n$, denoted by $A$, we can find a sufficiently large $J$ such that
  \begin{equation*}
    \begin{split}
      \left(\sum_{k\in A}\langle k \rangle^{\frac{sq}{1-\alpha}}\langle k\rangle^{\frac{n\alpha}{1-\alpha}(1-1/p)q}\right)^{\frac{1}{q}}
      \sim &
      \left(\sum_{k\in A}\langle k \rangle^{\frac{sq}{1-\alpha}}\|\mathscr{F}^{-1}\eta_k^{\alpha}\|^q_{L^p}\right)^{\frac{1}{q}}
      \\
      \lesssim &
      \left(\sum_{k\in\mathbb{Z}^n}\langle k \rangle^{\frac{sq}{1-\alpha}}\|\Box_k^{\alpha}f_J\|_{L^p}^q \right)^{\frac{1}{q}}
      \sim
      \|f_J\|_{M_{p,q}^{s,\alpha}}\lesssim \|f_J\|_{L^1}\lesssim  1.
    \end{split}
\end{equation*}
By the arbitrary of $A$, we actually obtain
  \begin{equation*}
      \left(\sum_{k\in \mathbb{Z}^n}\langle k \rangle^{\frac{sq}{1-\alpha}}\langle k\rangle^{\frac{n\alpha}{1-\alpha}(1-1/p)q}\right)^{\frac{1}{q}}
      \lesssim 1.
\end{equation*}
which yields that $s<n\alpha(1/p-1)+n(1-\alpha)(-1/q)=n\alpha(1/p-1)+n(1-\alpha)A(1,q)$.
\end{proof}

By the same method used to prove Proposition \ref{Proposition, Necessity of local Hardy to mod.}, we obtain
\begin{equation}\label{for proof, 13}
  s\leqslant n\alpha(1/p-1)+n(1-\alpha)A(1,q).
\end{equation}
Then, we use Proposition \ref{proposition, necessity of L1 to mod} to conclude that inequality (\ref{for proof, 13}) must be strict when $q\neq \infty$.

Next, we turn to the sufficiency part.
When $q=\infty$, we have $p\geqslant 1$ and $s\leqslant n\alpha(1/p-1)$. Using Young's inequality, we deduce that
\begin{equation}
  \begin{split}
    \|f\|_{M_{p,\infty}^{s,\alpha}}&=\sup_{k\in\mathbb{Z}^n}\|\Box_k^{\alpha}f\|_{L^p} \langle k \rangle^{\frac{s}{1-\alpha}}
                                   \lesssim
                                   \sup_{k\in\mathbb{Z}^n}\|\mathscr{F}^{-1}\eta_k^{\alpha}\|_{L^p} \langle k \rangle^{\frac{s}{1-\alpha}} \|f\|_{L^1}
                                   \\
                                   &\lesssim
                                   \sup_{k\in\mathbb{Z}^n} \langle k\rangle^{\frac{n\alpha}{1-\alpha}(1-1/p)}
                                   \langle k \rangle^{\frac{s}{1-\alpha}} \|f\|_{L^1}
                                   \lesssim
                                   \|f\|_{L^1}.
  \end{split}
\end{equation}
For $q<\infty$, taking $\epsilon$ to be a fixed positive number,
observing that $n(1-\alpha)/q+n\alpha(1/p-1)+n(1-\alpha)A(1,q)-\epsilon<n\alpha(1/p-1)$ , we use Lemma \ref{Lemma, embedding of alpha modulation spaces} to deduce
\begin{equation}
  \|f\|_{M_{p,q}^{n\alpha(1/p-1)+n(1-\alpha)A(1,q))-\epsilon,\alpha}}
  \lesssim
  \|f\|_{M_{p,\infty}^{n\alpha(1/p-1),\alpha}}
  \lesssim
  \|f\|_{L^1}.
\end{equation}

\subsection{Proof of Theorem \ref{theorem, sharpness of embedding from alpha-modulation spaces to L1}.}

By the same method used in the proof of Proposition \ref{Proposition, Necessity of mod to local Hardy.},
we can verify the necessity of Theorem \ref{theorem, sharpness of embedding from alpha-modulation spaces to L1}.
On the other hand, using Theorem \ref{theorem, sharpness of embedding from alpha-modulation spaces to local Hardy spaces} and the fact that $h^1\subset L^1$,
  we get the sufficiency of Theorem \ref{theorem, sharpness of embedding from alpha-modulation spaces to L1}.

\section{Embedding between $L^{\infty}$ and $\alpha$-modulation spaces}
\subsection{Proof of Theorem \ref{theorem, sharpness of embedding from L-infinity to alpha-modulation spaces}}
First, we give the proof for necessity.
 By the same method used in the proof Proposition \ref{Proposition, Necessity of local Hardy to mod.},
 we deduce $1/p\leqslant 1/\infty$, which yields $p=\infty$. We also have $s\leqslant n(1-\alpha)A(\infty,q)$ with strict inequality when $q\neq \infty$.

Next, we turn to the sufficiency part.
For $q=\infty$, we have $s\leqslant 0$.
Thus, we have
\begin{equation*}
  \|f\|_{M_{\infty,\infty}^{s,\alpha}}\leqslant \|f\|_{M_{\infty,\infty}^{0,\alpha}}=\sup_{k\in \mathbb{Z}^n}\|\Box_k^{\alpha}f\|_{\infty}\leqslant \sup_{k\in \mathbb{Z}^n}\|f\|_{\infty} \leqslant \|f\|_{L^{\infty}}.
\end{equation*}
When $q\neq\infty$ and $s<-n(1-\alpha)/q$, by Young's inequality, we have
\begin{equation}
  \begin{split}
    \|f\|_{M_{\infty,q}^{s,\alpha}}
                       &=
                       \left(\sum_{k\in\mathbb{Z}^n}\|\Box_k^{\alpha}f\|_{L^{\infty}}^q \langle k \rangle^{\frac{sq}{1-\alpha}} \right)^{\frac{1}{q}}
                       \lesssim
                       \left(\sum_{k\in\mathbb{Z}^n} \langle k \rangle^{\frac{sq}{1-\alpha}} \right)^{\frac{1}{q}} \|f\|_{L^{\infty}}
                       \lesssim
                       \|f\|_{L^{\infty}},
  \end{split}
\end{equation}
where we use ${sq}/(1-\alpha)<-n$ in the last inequality.
\subsection{Proof of Theorem \ref{theorem, sharpness of embedding from alpha-modulation spaces to L-infinity}.}
First, we give the proof for necessity.
By the structure of $L^{\infty}$,
we establish the following proposition for restriction of $M_{p,q}^{s,\alpha}\subset L^{\infty}$.

\begin{proposition}\label{proposition, necessity of mod to L-infty}
Let $0<p\leqslant \infty$, $q>1$, then $M_{p,q}^{s,\alpha}\subset L^{\infty}$ implies
\begin{equation*}
  l_q^{s+n\alpha(1-1/p),\alpha}\subset l_1^{n\alpha,\alpha}.
\end{equation*}
\end{proposition}
\begin{proof}
Take $f\in\mathscr{S}$ to be a nonzero smooth function whose Fourier transform has small support, such that $f(0)=1$,
$\Box_k^{\alpha}f_k=f_k$ and $\Box_l^{\alpha}f_k=0$ if $k\neq l$,
where we denote $\widehat{f_k}(x)=\widehat{f}\left(\frac{\xi-\langle k\rangle^{\frac{\alpha}{1-\alpha}} k}{\langle k\rangle^{\frac{\alpha}{1-\alpha}}}\right)$.
For a truncated (only finite nonzero items) nonnegative sequence $\vec{a}=\{a_k\}_{k\in \mathbb{Z}^n}$,
we define
\begin{equation}
  F(x)=\sum\limits_{k\in \mathbb{Z}^n}a_kf_k.
\end{equation}
By a direct computation, we have
\begin{equation*}
    \|F\|_{M_{p_2,q_2}^{s_2,\alpha}}
    \sim
    \|\vec{a}\|_{l_{q_2}^{n\alpha(1-1/p_2)+s_2,\alpha}}.
\end{equation*}
On the other hand, observing that $F(x)=\sum\limits_{k\in \mathbb{Z}^n}a_k\langle k\rangle^{\frac{\alpha n}{1-\alpha}}e^{2\pi i\langle k\rangle^{\frac{\alpha}{1-\alpha}} k\cdot x}f(\langle k\rangle^{\frac{\alpha}{1-\alpha}}x)$,
we have
\begin{equation*}
  \|F\|_{L^{\infty}}\geqslant F(0)=\sum_{k\in \mathbb{Z}^n}a_k\langle k\rangle^{\frac{\alpha n}{1-\alpha}}f(0)=\sum_{k\in \mathbb{Z}^n}a_k\langle k\rangle^{\frac{\alpha n}{1-\alpha}}=\|\vec{a}\|_{l_1^{n\alpha,\alpha}}.
\end{equation*}
Thus, we have
\begin{equation*}
  \|\vec{a}\|_{l_1^{n\alpha,\alpha}}
  \lesssim
  \|F\|_{L^{\infty}}
  \lesssim
  \|F\|_{M_{p,q}^{s,\alpha}}
  \sim
  \|\vec{a}\|_{l_{q}^{n\alpha(1-1/p)+s,\alpha}}.
\end{equation*}
By the arbitrary of $\vec{a}$, we obtain the desired conclusion.
\end{proof}
Using Proposition \ref{proposition, necessity of mod to L-infty} and Lemma \ref{lemma, Sharpness of embedding, for alpha decomposition}, we deduce that
$s\geqslant n\alpha/p$ for $1\leqslant 1/q$, and $s>n\alpha/p+n(1-\alpha)(1-1/q)$ for $1>1/q$, which is just the conclusion.

Now, we turn to the proof of sufficiency.
In fact,
by Lemma \ref{Lemma, embedding of alpha modulation spaces}, we have that
\begin{equation}
  \|f\|_{L^{\infty}}\leqslant \sum_{k\in\mathbb{R}^n}\|\Box_k^{\alpha}f\|_{L^{\infty}}
                    =         \|f\|_{M_{\infty,1}^{0,\alpha}}
                    \lesssim  \|f\|_{M_{p,q}^{s,\alpha}}.
\end{equation}
where $s\geqslant n\alpha/p$ for $1\leqslant 1/q$, or $s>n\alpha/p+n(1-\alpha)(1-1/q)$ for $1>1/q$.

\subsection*{Acknowledgements} This work was supported by the National Natural Foundation of China (No. 11601456).

\end{document}